\documentclass[11pt]{iopart}

\pdfoutput=1

\usepackage{url}
\usepackage{mystyle}
\usepackage[utf8]{inputenc}

\newcommand{\Ceremade}{CEREMADE, Universit\'e Paris-Dauphine, CNRS, Place du Marechal De Lattre De Tassigny, 75775 Paris 16, FRANCE}
\newcommand{\INRIA}{INRIA, MOKAPLAN, 2 rue Simone Iff, Paris, FRANCE}
\newcommand{\CMAP}{CMAP, Ecole Polytechnique, CNRS, 91128 Palaiseau, FRANCE}
\newcommand{\CNRS}{CNRS}
\newcommand{\Cam}{DAMTP, University of Cambridge, Wilberforce Road, Cambridge CB3 0DZ, UK}

\usepackage[style=numeric-comp,natbib=true,backend=bibtex,maxbibnames=99]{biblatex}
\bibliography{mybiblio.bib}
\AtEveryBibitem{\clearlist{url}}
\AtEveryBibitem{\clearlist{doi}}
\renewcommand\cite{\citep}

\begin{document}

\title[Geometric properties of solutions to the total variation denoising problem]{Geometric properties of solutions \\ to the total variation denoising problem}

\author{Antonin Chambolle$^1$, Vincent Duval$^2$, Gabriel Peyr{\'e}$^{3,4,2}$, Clarice Poon$^{4,2,5}$}

\address{$^1$ \CMAP}
\address{$^2$ \INRIA}
\address{$^3$ \CNRS}
\address{$^4$ \Ceremade}
\address{$^5$ \Cam}

\ead{antonin.chambolle@cmap.polytechnique.fr,vincent.duval@inria.fr\\gabriel.peyre@ceremade.dauphine.fr,cmhsp2@cam.ac.uk}



\begin{abstract}
This article studies the denoising performance of total variation (TV) image regularization. More precisely, we study geometrical properties of the solution to the so-called Rudin-Osher-Fatemi total variation denoising method.  
%
%
The first contribution of this paper is a precise mathematical definition of the ``extended support'' (associated to the noise-free image) of TV denoising. It is intuitively the region which is unstable and will suffer from the staircasing effect. We highlight in several practical cases, such as the indicator of convex sets, that this region can be determined explicitly.
Our second and main contribution is a proof that the TV denoising method indeed restores an image which is exactly constant outside a small tube surrounding the extended support. The radius of this tube shrinks toward zero as the noise level vanishes, and are able to determine, in some cases, an upper bound on the convergence rate. 
For indicators of so-called ``calibrable'' sets (such as disks or properly eroded squares), this extended support matches the edges, so that discontinuities produced by TV denoising cluster tightly around the edges. In contrast, for indicators of more general shapes or for complicated images, this extended support can be larger.  
Beside these main results, our paper also proves several intermediate results about fine properties of TV regularization, in particular for indicators of calibrable and convex sets, which are of independent interest. 
\end{abstract}



\section{Introduction}

The total variation (TV) denoising method was introduced by Rudin, Osher and Fatemi in~\cite{rudin1992nonlinear}. It is one of the first proposed non-linear image restoration method, and had an enormous impact on shaping modern imaging sciences. Despite being quite old, this method is still routinely used today, and its popularity probably stems from both its simplicity and its ability to restore ``cartoon-looking'' images. While being far from the state of the art for denoising in terms of performance (see Section~\ref{sec-pw} for some more recent works), it is still featured as a benchmark in most papers being published on image restoration. 

\subsection{Total Variation Denoising}

The total variation a function $u \in \LDD$ is defined as 
\begin{align}
  J(u)&\eqdef \int_{\RR^2} \abs{Du} 
  \eqdef \sup \enscond{\int_{\RR^2}u \divx z }{z\in \Cder{1}_c(\RR^2,\RR^2), \normi{z}\leq 1 }.
\end{align}

Given some noisy input function $f$, following~\cite{rudin1992nonlinear}, we are interested in the total variation denoising problem
\begin{equation}
  \min_{u\in \LDN} \la J(u)+ \frac{1}{2} \normLdeux{u-f}^2.
\tag{$\Pp_\la(f)$}\label{eq-rof}
\end{equation}
Here, $\la>0$ is the regularization parameter, and it should adapted by the user to the noise level. 

The goal of this paper is to study the ability to restore the geometrical structures (in particular the edges) of some (typically unknown) noise-free function $f$ by solving $\Pp_\la(f+w)$, i.e. by applying TV regularization to the input noisy image $f+w$. Here $w$ accounts for some additive noise in the image formation process, and is assumed to have a finite $L^2$ norm $\normLdeux{w}$.

\subsection{Previous Works}
\label{sec-pw}

\paragraph{Image restoration.}

The TV denoising method, often referred to as the Rudin-Osher-Fatemi (ROF) model, was introduced in~\cite{rudin1992nonlinear}. Its basic properties (including the existence and uniqueness of the solution) are derived in~\cite{ChambolleLions}. We refer to~\cite{Chambolle10anintroduction} for an introduction to this model and an overview of its numerous applications in image processing. A thorough study of its properties can be found in~\cite{Allard1,Allard2}.
It is important to realize that TV is far beyond the state of the art in imaging sciences, and we refer to recent works such as~\cite{Portilla03,BuadesCM05,Aharon06,Dabov07} that obtain superior denoising performance on natural images by exploiting more complex and involved regularizers and statistical models. 



Beyond denoising, TV methods have been used successfully to solve a wide range of ill-posed inverse problems, see for instance~\cite{AcarVogel,ChanMarquina00,ChaventCOV,Malgouyres02}.
Following the work of Meyer~\cite{Meyer}, TV regularization in conjunction to a norm dual of TV (favoring oscillations) is used to separate texture from structure, see for instance~\cite{Aujol}.  
In a finite dimensional setting (using a discretization of the gradient operator), TV methods have been used to solve compressed sensing, where the linear operator is randomized~\cite{NeedelWard12,Poon15} to obtain accurate reconstructions when the number of random samples is nearly proportional to the number of the discretized edges. 

\paragraph{Jump sets stability.}

The use of non-smooth (possibly non-convex) regularizations to restore edges and promote sharp features has been advocated by Mila Nikolova. She provided in a series of papers a detailed analysis of a general class of regularization schemes which admit piecewise smooth solutions, see for instance~\cite{NikolovaStrong00}. 
In the case of the TV regularization, this analysis can be refined. 
Explicit solutions are known, mostly in 1-D and for radial 2-D functions (see for instance~\cite{StrongExplicit96}), as well as for indicators of convex sets in the plane~\cite{alter2005evolution,Allard3}. They suggest that TV methods indeed do maintain sharp features.
A landmark result is the proof in~\cite{casdiscont07} that total variation regularization does not introduce jumps, i.e. the ``jump set'' of the solution of~\eqref{eq-rof} is included in the one of the input $f$. A review of this result and extensions can be found in~\cite{Valkonen15}.

These results are however of little interest when $f$ is replaced by a noisy function $f+w$ (which is the setting of practical use of the method), since the noise $w$, which is only assumed to be in $L^2$, might introduce jumps everywhere. It is actually the presence of this noise which is responsible for the ``staircasing'' effect, which 
creates spurious edges in flat area. Properties of this staircasing are studied in 1-D~\cite{Ring00} and in higher dimension in~\cite{Jalalzai2015}.  It is the purpose of the present paper to fill this theoretical gap by analyzing the impact of the noise on the jump set of the solution to $\Pp_\la(f+w)$, when both $\normLdeux{w}$ and $\la$ are not too large. 

\paragraph{Calibrable and Cheeger sets.} 

Of particular importance for the analysis of TV methods are indicator functions of sets, and their behavior under the regularization.
Indicator functions which are invariant (up to a rescaling) under TV denoising define so-called ``calibrable'' sets. These sets play the role of ``stable'' sets and one expects the corresponding edges to be well restored by TV denoising, a statement which is made precise in the present paper. 
We refer to section~\ref{sec-calibrable} for a detailed description of these sets and their basic properties. 
An important result is the full characterization of convex calibrable sets in~\cite{altercalib05}.
The notion of a calibrable set is closely related to the one of eigenvectors of the curvature operator, which informally reads $\divx(\tfrac{D u}{|Du|})$, and is also known as the $1$-Laplacian, see~\cite{Kawohl07}. Indeed, indicators of calibrable sets are eigenvectors of this operator~\cite{BellettiniEigen05}. These eigenvectors can be used for image processing purposes, as advocated in~\cite{Benning13}.
The study of fine geometrical properties of TV minimizers is thus deeply linked with geometric measure theory and in particular sets of finite perimeters~\cite{ambcasmas99,maggi2012sets}. In particular, the construction of calibrable sets is related to minimal surface problems~\cite{giusti1977minimal} and capilarity problems~\cite{Kor93}.
Calibrable sets are also related to Cheeger sets, which are subsets of a given set minimizing the ratio of perimeter over area. These Cheeger sets are useful to construct the solution of the TV denoising problem. Cheeger sets associated to a given convex sets are unique~\cite{casuniq07,alteruniq09,kawohl06}, and can be approximated using either $p$-Laplacian~\cite{kawohlnovaga} or strictly convex penalizations~\cite{buttmaximal07} to recover an unique maximal set, which can in turn be computed numerically~\cite{carlierpeyre}. 

\paragraph{TV flow.}

While our paper studies variational problems, a closely related denoising method is obtained by solving the PDE obtained as a gradient flow of $J$, see~\cite{beltvflow02} for a formal definition. In this setting, the evolution time $t$ plays the role of $\la$. 
This TV flow can be shown in 1-D, for characteristic of convex sets and for radial functions to be equivalent to the TV regularization~\eqref{eq-rof}, see~\cite{Ring00,Briani2011,BroxEquiv,Jalalzai2015}.
All the results available for the variational formulation~\eqref{eq-rof} have equivalent in the PDE setting, such as for instance explicit solutions for the indicators of convex sets~\cite{alter2005evolution} and the evolution of the jump-set~\cite{CasellesJumpFlow}. Some of these results have been extended to more general PDE's, see~\cite{andreu2004parabolic}.

\paragraph{1-D setting and statistical estimation.}

1-D TV denoising, sometimes referred to as the ``taut string method''~\cite{mammen1997}, is a method of choice to perform statistical analysis of time series and in particular to detect jumps and transitions. 
In 1-D, TV flow and TV regularization are known to be equivalent~\cite{Ring00,Briani2011}. 
In the special 1-D case, it is possible to compute exactly the solution on a grid of $P$ points in $O(P^2)$ operations using a dynamic programming method~\cite{Condat13,Dumbgen2009,Grasmair2006,HinterbergerTube}.
Similarly to wavelet thresholding estimators, 1-D TV denoising is known to achieve asymptotic optimal estimation results~\cite{mammen1997}. This optimality is however measured in term of $L^2$ error, which does not provide geometric information about the location of jumps. A more precise analysis of the distribution of the jumps is provided in~\cite{davies2001}. This analysis is however probabilistic and does not extend to higher dimensions, whereas we targets a deterministic geometric analysis in 2-D (although some of our results cover the general $N$-dimensional case).

\paragraph{Inverse problem and source condition.}

The systematic study of noise stability of regularization schemes relies on the so-called source-condition~\cite{ScherzerBook09}, which reads in the simple denoising setting that $\partial J(f)$ should be non-empty (see Section~\ref{sec-subdiff} for a primer on the total variation sub-differential $\partial J$). For non-smooth regularizations over Banach spaces, this study started with the seminal paper of Burger and Osher~\cite{burger2004convergence} who show that this source condition implies stability of the solution according to the Bregman divergence associated to $J$. This Bregman measure of stability is however quite weak, and in particular it does not lead to a precise geometric characterization of the restored jump set. Our analysis can be seen as a generalization and refinement of this approach, as highlighted in Section~\ref{sec:ourapproach}. Note that under a non-degeneracy condition, namely that $0$ is in the relative interior of $\partial J(f)$, it is possible to state much stronger results, as detailed in the book~\cite{ScherzerBook09} for $\ell^1$-based methods. These results however do not cover the TV regularization and can only be applied to discretized versions of TV regularization problems, see~\cite{VaiterPDF13}.

\paragraph{Numerical algorithms.}

While this is not the topic of this article, let us note that the discretization (often using finite differences) and the numerical resolution of~\eqref{eq-rof} is notoriously difficult, in large part because of the non-smoothness of the TV functional $J$. 
Early algorithms rely on various smooth approximations of $J$~\cite{VogelOman,ChambolleLions}. 
The dual projected gradient method proposed by~\cite{chambolle2004algorithm} started a wave of activity on the use of first order proximal splitting schemes to solve~\eqref{eq-rof} with a provably convergent scheme, see for instance~\cite{BeckTeboulle} for accelerated first order schemes. 
Another option is to solve exactly the denoising problem using graph-cuts methods~\cite{Hochbaum2001}, see also~\cite{KolmogorovZabin,DarbonSigelle,ChambolleDarbon2009} and the references therein. These algorithms work however only for the anisotropic total variation, and thus do not cover our $J$ functional, which is the isotropic total variation. 
Let us also recall that TV methods, and their discretizations, are intimately linked with iterative non-linear filterings, and in particular local median filters, see~\cite{Buades2006}. 

\subsection{Contributions}

\paragraph{Level lines in the low noise regime.}
Let us first stress the fact that our analysis focusses on regimes where the noise and the regularization parameter are small. It is not very difficult to see that, as $\la\to 0^+$ and $\normLdeux{w}\to 0^+$, the solution $\ulaw$ to $\Pp_\la(f+w)$ converges towards $f$ in the $L^2$ topology. Our goal is to describe this convergence more precisely: \textit{is it possible to say that the level lines of $\ulaw$ converge to those of $f$? In what sense? Morevoer, does the support of $D\ulaw$ converge towards the support of $Df$?}
Those questions are all the more important as it is widely acknowledged in image analysis theory that the shape information of an image is contained in the level sets of an image~\cite{wertheimer1923,serra1982}, determined in particular by their boundary. 

To assess the support stability of the method with respect to that matter, and in particular to study its ability to restore edges, it is necessary to make stronger assumption on the noise level. Whereas in~\cite{duvalpeyre13} we considered low noise regimes in which $\la \to0^+$ and ${\normLdeux{w}}/{\la}\leq C$ for some well-chosen constant $C>0$, here, to obtain strongest results, we assume the stronger condition ${\normLdeux{w}}/{\la}\to 0^+$ as $\la\to 0^+$.  


\subsubsection{Our approach -- the minimal norm certificate}\label{sec:ourapproach}
A common approach to studying the stability properties of a variational problems is  by analysis of the source condition. To explain our approach, we  first recall a result of Burger and Osher \cite{burger2004convergence} which provides a link between the source condition and stability of regularized solutions:  Given  $f\in \LDD$ with finite total variation, suppose that the source condition holds. i.e.  there exists some $v\in \partial J(f)$ such that $v = -\divx z$ for some $z\in \mathrm{L}^\infty(\RR^2, \RR^2)$ with $\norm{z}_{L^\infty}\leq 1$. Note that elements in $\partial J(f)$ are often referred to as \emph{dual certificates}. Let $T\subset \RR^2$ and $\delta \in (0,1)$ be such that $\abs{z(x)}<1- \delta$  for a.e. $x\not\in T$. Then,
 $\ulaw$, the solution to $\Pp_\la(f+w)$ satisfies
$$
(1-\delta)\int_{\RR^2\setminus T} \abs{D \ulaw}\leq  \frac{\normLdeux{w}^2}{2\lambda} +\frac{\lambda  \norm{v}_{L^2}^2}{2} + \normLdeux{w} \norm{v}_{L^2}.
$$

While this result informs us that the variation of $\ulaw$ is concentrated in the region $T$, it does not provide any information on the regions where $D \ulaw$ is identically zero and no information is given about how the support of $D \ulaw$ differs from the support of $D f$. 

Instead of studying \textit{any} $v\in \partial J(f)$, in this paper, we shall study the minimal norm certificate
$$
\voo \eqdef \mathop{\argmin}\enscond{\norm{v}_{L^2}}{ v\in\partial J(f)}.
$$

The minimal norm certificate was first proposed in \cite{duvalpeyre13} for  studying the support of solutions to the sparse spikes deconvolution problem using total variation of measures regularization, but in this particular framework of denoising, it is also known as the \textit{minimal section} of $\partial J(f)$ \cite{scherzer2008variational}.
Although dual certificates have been widely used  to derive stability properties of solutions to the sparse spikes deconvolution problem in terms of the $L^2$ norm, see for instance \cite{Grasmair-cpam}, the novelty of the minimal norm certificate (which is itself a dual certificate) is that it additionally addresses support stability questions such as the number and the location of the recovered diracs.

In this paper, we  follow the same philosophy: Similarly to the problem of sparse spikes deconvolution, we show that the minimal norm certificate naturally gives rise to the notion of an extended support, which in turn, governs the support of the regularized solution in the low noise regime.
Unlike previous works, our analysis is carried out for this very specific dual certificate and in doing so, we are able to characterize the support stability of the total variation denoising problem.

\subsubsection{Our main contribution}

\paragraph{The extended support.}

Based on the minimal norm certificate, we define the extended support $\ext(Df)$
of a function  $f\in \LDD$ with bounded variation when the source condition is satisfied. 
Intuitively, it is the region that suffers from gradient support instabilities in the low noise regime.
The statement is made precise in the main result of this paper, Theorem~\ref{thm:spt_stability},  where we prove that given any tube around the extended support, there exists $\la_0,\al_0>0$ such that the support of $D u_{\la,w}$ is contained inside this tube for all $(\la,w)\in D_{\la_0,\al_0}$. Furthermore, the radius of this tube converges to zero as the noise level converges to zero. In particular, given sequences $w_n\in L^2(\RR^2)$ and $\la_n\in \RR_+$ such that
${\normLdeux{w_n}}/{\la_n}\to 0$ as $n\to +\infty$, the conclusion of our main result is that
\begin{equation*}
  \supp(Df)\subseteq \liminf_{n\to +\infty} \supp(D\un) \subseteq \limsup_{n\to+\infty} \supp(D\un)\subseteq \ext(Df).
\end{equation*}

Explicit examples of the extended support are given for indicator functions on calibrable sets and convex sets with smooth boundaries, and in particular, for these examples, our definition of the extended support is in fact tight.  Moreover, when denoising the indicator function of a calibrable set $C$, the support of regularized solutions to the TV denoising problem will cluster around $\partial C$.

\paragraph{Stability estimates in the absence of the source condition.}
Section \ref{sec:no_source_cond} discusses stability analysis for cases where the source condition is not satisfied, i.e. $\partial J(f) = \emptyset$. One important class of functions which this covers are the indicator functions  on convex sets with nonsmooth boundary, such as the  square. To our knowledge, there were no previous studies on stability analysis in the absence of the source condition and hence, no stability guarantees for even the simple case of denoising the indicator function of a square. Although  in this case, the minimal norm certificate  is not defined, we show in Theorem \ref{thm:union_cvx_sets} that the techniques developed in the analysis of the minimal norm certificate can be adapted to such special cases to derive stability estimates for general convex sets.

\paragraph{Convergence rates.}

We stress that  via the approach of \cite{burger2004convergence}, characterization of the regions where the variation of $\ulaw$ is small is possible only when the source condition holds \emph{and} there is precise knowledge of the extremal points and decay of some vector field $z\in \mathrm{L}^\infty(\RR^2, \RR^2)$ for which $v=-\divx z \in \partial J(f)$. In general, this vector field  is not unique and such precise characterization is a difficult problem. 
In contrast, via our approach, explicit knowledge of the vector field associated with the minimal norm certificate is not essential, and in fact, the definition of the extended support is dependent only $v_0$ and not on the vector fields $z$ for which $v_0= -\divx z$.  

Nonetheless, in the special cases where the vector field associated with $v_0$ is known, we provide in Theorem \ref{thm:stab_w_vec_field} an explicit upper bound on the rate of shrinkage of the tube around the extended support with respect to the decay of the noise level. For the indicator function on a calibrable set $C$ with $\Cder{2}$ boundary,  we describe an explicit construction of the vector field $z_0$ associated with the minimal norm certificate with $\abs{z_0}<1$ on all compact subsets of $\RR^2\setminus \partial C$.
Therefore, our main result can be seen as a refinement of the work of \cite{burger2004convergence} and can be applied in much greater generality.

\subsection{Outline of the paper}
Section \ref{sec:prelim} recalls some essential tools which will be used throughout this paper. Section \ref{sec:duality} introduces the dual formulation of \eqref{eq-rof} and defines the minimal norm certificate. Explicit examples of the minimal norm certificate are also given. Based on the existence of the minimal norm certificate, Section \ref{sec:proplev} derives some geometric properties of the level sets of solutions of \eqref{eq-rof} in the low noise regime. The definition and examples of the extended support can be found in Section \ref{sec:extended}.
The key results are presented in Sections \ref{sec:stab_extended_spt}, \ref{sec:no_source_cond} and \ref{sec:stab_grad}. Section \ref{sec:stab_extended_spt} presents the main result, which describe support stability with respect to the extended support in the presence of the source condition. Section   \ref{sec:no_source_cond} decribes how our main result can be adapted for the analysis of support stability in the absence of the source condition; in particular, we provide analysis for the special case of denoising indicator functions on unions of convex sets. Section \ref{sec:stab_grad} presents a refinement of our main theorem in the case where the vector field associated with the minimal norm certificate is known. Furthermore, in Section \ref{sec:calibrable}, we describe the behaviour of the vector field associated with the minimal norm certificate of indicator functions on calibrable sets.
Finally, some numerical examples are presented in Section \ref{sec:numerics} for the illustration of our theoretical results. 


\section{Preliminaries}
\label{sec:prelim}
This section recalls some essential results which are applied throughout this paper.

\subsection{Set convergence}
\label{sec:prelim-setcv}
We shall use the notion of \textit{Painlev\'e-Kuratowski} set convergence (see~\cite{rockafellarwets} for more detail). Given a sequence of sets $\{S_n\}_{n\in\NN}$, $S_n\subseteq \RR^2$,  let us define the outer (resp. inner) limit of $\{S_n\}_{n\in\NN}$ as
\begin{align}
  \limsup_{n\to +\infty} S_n\eqdef\enscond{x\in\RR^2}{\liminf_{n\to +\infty} \dist(x,S_n)=0}, \\
  \mbox{(resp.)}\quad  \liminf_{n\to +\infty} S_n\eqdef\enscond{x\in\RR^2}{\limsup_{n\to +\infty} \dist(x,S_n)=0}.
\end{align}
It is clear that $\liminf_{n\to +\infty} S_n\subseteq \limsup_{n\to +\infty} S_n$. Moreover, those two sets are closed. We say that $S_n$ converges towards $S\subseteq \RR^2$, \textit{i.e.\ } $\lim_{n\to+\infty} S_n=S$, if
\begin{equation}
  \liminf_{n\to +\infty} S_n =S =\limsup_{n\to +\infty} S_n.
\end{equation}

If the sequence $S_n$ is bounded (there exists $R>0$ such that $S\subseteq B(0,R)$ and $S_n\subseteq B(0,R)$ for all $n$ large enough), then the \textit{Painlev\'e-Kuratowski} convergence is equivalent to the so-called Hausdorff convergence, that is,
\begin{equation}
  \lim_{n\to+\infty} \sup_{x\in S\cup S_n}\abs{\dist(x,S_n)-\dist(x,S)}=0.
\end{equation}

\subsection{Functions with bounded variation and sets with finite perimeter}
We briefly recall some properties of functions of bounded variations and sets of finite perimeter. We refer the reader to \cite{Ambrosio,maggi2012sets} for a comprehensive treatment of the subject.

\paragraph{Total variation, perimeter.}
Given $u\in L^1_{loc}(\RR^2)$, its total variation is equal to
\begin{align*}
\int_{\RR^2} \abs{Du} \eqdef \sup \enscond{\int_{\RR^2}u \divx z }{z\in \Cder{1}_c(\RR^2,\RR^2), \normi{z}\leq 1 }.
\end{align*}

If $J(u)<+\infty$, we say that $u$ has bounded variation.
The mapping $u\mapsto J(u)$ is lower semi-continuous  with respect to the $L^1_{loc}(\RR^2)$ topology (hence for the $L^2$ topology).

If $E\subseteq \RR^2$ is a measurable set, we denote by $|E|$ its 2-dimensional Lebesgue measure. The set $E$ is said to be of finite perimeter if $J(\bun{E})<+\infty$, where $\bun{E}$ is the indicator function of $E$. Its perimeter is defined as $P(E)=\int_{\RR^2} |D\bun{E}|$. For a Borel set $S\subseteq \RR^2$, $\Hh^1\llcorner S$ denotes the 1-dimensional Hausdorff measure restricted to $S$, namely $\Hh^1\llcorner S(A)=\Hh^1(A\cap S)$. 

The reduced boundary of $E$ is defined as
\begin{equation}
  \partial^* E\eqdef \enscond{x\in \supp \abs{D\bun{E}}}{\nu_E(x)\eqdef \lim_{r\to 0^+}\frac{-D\bun{E}(B(x,r)}{\abs{D\bun{E}(B(x,r))}}\mbox{ exists and }\abs{\nu_E(x)}=1}. 
\end{equation}
The vector $\nu_E(x)$ is the measure theoretic outer unit normal to $E$. When the context is clear, we shall write $\nu$ instead of $\nu_E$. Moreover, $D\bun{E}=-\nu_E\Hh^1\llcorner \partial^*E$, and $\abs{D\bun{E}}(A)=\Hh^1(\partial^*E \cap A)$ for all open set $A\subseteq \RR^2$.

In the following, we use the construction in~\cite[Prop.~3.1]{giusti1977minimal} so as to always consider a Lebesgue representative of $E$ such that for all $x$ in the topological boundary $\partial E$, $0<\abs{E\cap B(x,r)}<\abs{B(x,r)}$. Then, with this representative,
\begin{equation*}
  \supp D\bun{E}=\overline{\partial^*E}=\partial E.
\end{equation*}

The area and the perimeter are related by the so-called isoperimetric inequality: for any Lebesgue measurable set $E\subseteq \RR^2$, 
$$\cD\min\{\abs{E},\abs{\RR^N\setminus E}\}\leq (P(E))^2,$$ where  $\cD=4\pi$ is the isoperimetric constant.

\paragraph{Level sets and the coarea formula.}
The coarea formula relates the total variation of a function $f\in L^1_{loc}(\RR^2)$ and the perimeter of its level sets. Define the level sets of $f$ as
\begin{equation}\label{eq:lev_sets}
\begin{split}
  \F{t}\eqdef \enscond{x\in\RR^2}{f(x)\geq t} \mbox{ for } t\geq 0,\\
  \F{t}\eqdef \enscond{x\in\RR^2}{f(x)\leq t} \mbox{ for } t< 0.
  \end{split}
\end{equation}
It is clear that $\abs{\F{t}}<+\infty$ except possibly for $t=0$. Moreover, the family is monotone on $[0,+\infty)$ and $(-\infty,0)$ with
\begin{equation*}
  \F{t}=\bigcap_{0<t'<t}\F{t'} \mbox{ for $t>0$,} \quad  \F{t}=\bigcap_{0>t'>t}\F{t'} \mbox{ for $t<0$}.
\end{equation*}
We handle $0$ as a special case with $\F{0}=\RR^2\setminus\bigcup_{t'<0}\F{t'}$.
Now, given an open set $U\subseteq \RR^2$, the coarea formula states that if $J(f)<+\infty$ then
\begin{equation*}
  \int_{U}\abs{Df} =\int_{-\infty}^{+\infty} P(\F{t};U)\d t.
\end{equation*}
where $P(\F{t};U)\eqdef \abs{D\bun{E}}(U)$.

\subsection{Subdifferential of $J$}
\label{sec-subdiff}

In the following, unless otherwise stated, we use the $\LDD$ topology. The functional $J:\LDD\rightarrow \RR\cup \{+\infty\}$ is convex, proper lower semi-continuous. It is in fact the support function of the closed convex set
\begin{equation*}
  \enscond{\divx z }{z\in \XDD, \normi{z}\leq 1}\subseteq \LDD, 
\end{equation*}
where we defined 
\eq{
	\XD \eqdef \enscond{z\in \Linf(\RR^2, \RR^2) }{ \divx z\in\LDD  }.
}
As a result, it is possible to prove that
\begin{align}
  \label{eq:tvsubdiff}
  \partial J(0)&= \enscond{\divx z }{z\in \XDD, \normi{z}\leq 1},\\
\label{eq:tvsubdiffu}  \partial J(u)&= \enscond{v\in \partial J(0) }{\int_{\RR^N} uv = J(u)}.
\end{align}

Provided that $J(u)<+\infty$, $Du$ is a Radon measure, i.e  it is possible to evaluate $(z,Du)$ for all vector field $z\in C_c^0(\RR^2; \RR^2))$.
Following the construction by Anzellotti~\cite{Anzellotti}, it is possible to define $(z,Du)$ for less smooth $z$, namely $z\in \XD$ provided that $u\in \LDD$ and $J(u)<+\infty$. Given $\varphi\in \Cder{1}_c(\RR^2)$, define
\begin{align*}
  \langle (z,Du), \varphi \rangle= -\int_{\RR^2} u(x)\varphi(x) \divx z(x)dx - \int_{\RR^2} u(x)z(x)\cdot \nabla \varphi(x)dx.
\end{align*}
Then $(z,Du)$ is a Radon measure which is absolutely continuous with respect to $\abs{Du}$, with 
\begin{equation*}
  \abs{\langle (z,Du),\varphi\rangle} \leq \normi{\varphi}\norm{z}_{L^\infty(U)} \int_U \abs{Du},
\end{equation*}
for all $\varphi\in \Cder{1}_c(\RR^2)$ and $U\subset \RR^2$ open set such that $\supp(\varphi)\subset U$. Moreover, the following integration by parts holds
\begin{equation*}
  \int_{\RR^2}u\divx z = -\int_{\RR^2}(z,Du).
\end{equation*}

If $\theta(z,Du)$ denotes the Radon-Nikodym derivative of $(z,Du)$ with respect to $\abs{Du}$, we may also write write $\int_{\RR^2}(z,Du)=\int_{\RR^2}\theta(z,Du)\d \abs{Du}$.


\begin{rem}
If $u$ is smooth, then $(z,Du)$ can be interpreted as a (defined almost everywhere) pointwise inner product:
\begin{align*}
  \int_B (z,Du) = \int_B z(x) \cdot \nabla u(x)dx \quad \mbox{ for any Borel set } B \subseteq \RR^2.
\end{align*}

If $u$ is the characteristic function of set with finite perimeter $E\subset \RR^2$
\begin{align*}
 \int_E \divx z(x)dx = \int_E (z,-D\bun{E}).
\end{align*}
The question whether it is possible to give a pointwise meaning to $(z,-D\bun{E})$ is investigated in~\cite{bredies2012,ChaGolNov12a}.
In~\cite{bredies2012}, under some regularity assumption on $z$ (which holds if $\divx z \in \partial J(\bun{E})$), it is interpreted as $(z,-D\bun{E})=Tz\cdot \nu_E \Hh^1\llcorner \partial^*E$, where $Tz$ is the full trace of $z$ defined on $\Hh^1$-a.e. on $\partial^*E$~\cite{bredies2012}. In~\cite{ChaGolNov12a}, it is shown that (in dimension 2 or 3), if $\divx z \in \partial J(\bun{E})$, then every point of the reduced boundary $\partial^*E$ is a Lebesgue point of $z$, hence $(z,-D\bun{E})=z\cdot \nu_E \Hh^1\llcorner \partial^*E$.

In the general case, recalling that $\abs{D\bun{E}}= \Hh^1\llcorner \partial^*E$, we shall write
\begin{equation}
  \int_E \divx z(x)dx = \int_{\partial^*E} \theta(z,-D\bun{E})\d\Hh^{1},
\end{equation}
keeping in mind that in regular cases this amounts to $\int_{\partial^*E} z\cdot \nu_E\d\Hh^{1}$
\end{rem}

\begin{rem}
  That  enables us to interpret the optimality $\int_{\RR^2} uv = J(u)$ as an ``optimality $\abs{Du}$-almost everywhere'': 
\begin{align*}
  \label{eq:}
  \quad \int_{\RR^N} u\divx z = \int_{\RR^N} \abs{Du} 
  &\Leftrightarrow   -\int_{\RR^N} (z,Du) = \int_{\RR^N} \abs{Du} \\
  &\Leftrightarrow  0 = \int_{\RR^N} (1+\theta(z,Du))\diff \abs{Du},
\end{align*}
where $\theta(z,Du)$ is the Radon-Nikodym derivative of $(z,Du)$ with respect to $\abs{Du}$.
Since $\abs{\theta(z,Du)}\leq 1$, this implies that in fact the equality $(z,Du)=-\abs{Du}$ holds $\abs{Du}$-a.e.
Informally, recalling that $\normi{z}\leq 1$, this means that
\begin{equation*}
  z= -\frac{Du}{\abs{Du}}, \quad \abs{Du}-\mbox{almost everywhere.}
\end{equation*}

\textbf{In other words $z$ must be orthogonal to the level lines, and its saturation points contains the support of $Du$} (see also~\cite{bredies2012,ChaGolNov12a} for more rigorous statements).
\end{rem} 

\paragraph{Examples}
Let us examine two examples which can be found in~\cite{Meyer}.

\subparagraph{Characteristic function of a disc:}
Given $R>0$, consider the vector field 
\begin{equation}\label{eq-calib-disc}
  z(x)=\begin{cases}
    \frac{x}{R} & \text{if } \abs{x}\leq R, \\
    \frac{R}{\abs{x}^2}x & \text{ otherwise.}
  \end{cases}
\end{equation}
One may check that $z\in \Linf(\RR^2,\RR^2)$, $\divx z= \frac{2}{R}\bun{B(0,R)}\in \Linf(\RR^2,\RR^2)$, $\normi{z}\leq 1$, and $z\cdot \nu =1$ on $\enscond{x\in\RR^2}{\abs{x}=R}$. Hence  $\divx z\in \partial J(u)$ for $u=\bun{B(0,R)}$.

\subparagraph{Characteristic function of a square~:} Let $u= \bun{[0,1]^2}$ be the characteristic function of the unit square. It turns out that $\partial J(u)= \emptyset$. 
The argument provided in~\cite{Meyer} is the following. Assume that there exists $v\in \partial J(u)$ and let $z\in \Linf(\RR^2,\RR^2)$ be the corresponding vectorfield. We denote by $T_{\varepsilon}$ the triangle $\enscond{x\in \RR^2}{0\leq x_1\leq 1,\ 0\leq x_2\leq 1,\ x_1+x_2\leq \varepsilon}$ and by $\nu$ its outer unit normal (defined $\Hh^1$-a.e.). By the Gauss-Green theorem:
\begin{equation*}
  \int_{T_\varepsilon} \divx z  = \int_{\partial T_\varepsilon} \theta(z,-D\bun{T_\varepsilon})\diff \Hh^1.
\end{equation*}
Since $v=\divx z\in\Ldeux(\RR^2)$, the left term is upper-bounded by $\sqrt{\int_{T_\varepsilon} ({\divx z})^2}\sqrt{\abs{T_\varepsilon}}=\smallo{\varepsilon}$ whereas the right term is lower-bounded by $(2-\sqrt{2})\varepsilon$. This is a contradiction. Hence $\partial J(u)= \emptyset$

\subsection{Calibrable sets in $\RR^2$}
\label{sec-calibrable}

A remarkable family of elements of $\partial J(0)$ is the family of characteristic functions of sets $F$ such that $\la_F\bun{F}\in \partial J(\bun{F})$ for some $\la_F\in \RR$. This family of functions, known as the calibrable sets, will serve as a prime example in the illustration of our theoretical results.   In this section, we recall some key results about these functions.

\subsubsection{Sets that evolve at constant speed}
In \cite{beltvflow02}, the authors study on the total variation flow $\frac{\partial u}{\partial t}\in -\partial J(u)$, namely:
\begin{align}
  \frac{\partial u}{\partial t}= \divx \left(\frac{Du}{|Du|} \right).
  \label{eq-tvflow}
\end{align}
They prove existence and uniqueness of a ``strong solution'' (see \cite{beltvflow02}) for all initial data $u_0\in \LDD$, and existence and uniqueness of an ``entropy solution'' for $u_0\in L^1_{loc}(\RR^N)$.
In the second part of the paper, they characterize the bounded sets of finite perimeter $\Omega$ such that $u=\bun{\Omega}$ satisfies
\begin{align}
  -\divx \left(\frac{Du}{|Du|}\right)= \lambda_\Omega u, \quad \mbox{ where } \la_\Omega\eqdef \frac{P(\Omega)}{|\Omega|}.
\end{align}

Such sets are exactly the sets which evolve with constant boundary, i.e. such that $u(x,t)=\la(t)\bun{\Omega}(x)$, with $\la \geq 0$. Such sets are called calibrable. They are characterized by the fact that $\la_\Omega\bun{\Omega}\in \partial J(\bun{\Omega})$:

\begin{defn}[Calibrable sets]
  A set of finite perimeter $\Omega\subset \RR^2$ is said to be calibrable if, writing $v=\bun{\Omega}$, there exists a vector field $z\in L^{\infty}(\RR^2,\RR^2)$ such that $\|z\|_\infty \leq 1$
  and 
  \begin{align*}
    \int_{\RR^N} (z, Dv)=\int_{\RR^2} |Dv|,\\
    -\divx z =\la_\Omega v.
  \end{align*}
  In that case, we say that $z$ is a calibration for $\Omega$.
\end{defn}

\begin{rem} If $\la \bun{\Omega}\in \partial J(\bun{\Omega})$ for some $\la\in \RR$, then necessarily $\la=\la_\Omega$.
\end{rem}

\subsubsection{Characterization in $\RR^2$}
The following results characterize convex calibrable sets.
\begin{prop}[\protect{\cite{beltvflow02}}]
Let $C\subset \RR^2$ be a bounded set of finite perimeter, and assume that $C$ is connected. $C$ is calibrable if and only if the following three conditions hold:
\begin{enumerate}
  \item $C$ is convex;
  \item $\partial C$ is of class $C^{1,1}$;
  \item the following inequality holds:
    \begin{align}\label{eq:prelimcalibchar}
      \ess \sup_{p\in \partial C} \kappa_{\partial C} (p)\leq \frac{P(C)}{|C|}.
    \end{align}
\end{enumerate}
\end{prop}
\begin{prop}[\protect{\cite{beltvflow02}}]\label{prop:prelimcalibminimal}
Let $\Omega \subset \RR^2$ be a bounded set of finite perimeter which is calibrable. Then,
\begin{enumerate}
  \item The following relation holds:
    \begin{align*}
      \frac{P(\Omega)}{|\Omega|} \leq \frac{P(D)}{|D\cap \Omega|}, \quad \forall D\subseteq \RR^2, D \mbox{ of finite perimeter;}
    \end{align*}
  \item each connected component of $\Omega$ is convex.
\end{enumerate}
\end{prop}

\subsection{From the subdifferential to the level sets}
Let $f\in\LDD$, $J(f)<+\infty$, and $v\in \partial J(f)$. By definition of the subdifferential,
\begin{equation}\label{eq:prelim-subdiff}
  \forall g\in \LDD,\quad \int_{\RR^2}\abs{Dg} - \int_{\RR^2} vg \geq \int_{\RR^2}\abs{Df}-\int_{\RR^2} vf.
\end{equation}

In fact, using the coarea formula, one may reformulate that optimality property (see Proposition~\ref{prop:prelim-subdiff} below) as an optimality property of the level sets. That result is very similar to~\cite[Corollary~2.4]{kindermann2006denoising} but it requires a bit more care in our framework since the domain is $\RR^2$ and $v\in \LDD$.

The level sets of $f$ (resp. $g$), are denoted by $\{\F{t}\}_{t\in \RR}$ (resp. $\{\G{t}\}_{t\in \RR}$).

\begin{prop}\label{prop:prelim-subdiff}
  Let $f\in \LDD$, $J(f)<+\infty$, and $v\in \LDD$. The following conditions are equivalent.
\begin{itemize}
  \item[(i)] $v\in \partial J(f)$,
  \item[(ii)]  $v\in \partial J(0)$ and the level sets of $f$ satisfy
    \begin{align}\label{eq:prelim-setgeomcalib}
      \forall t>0,\quad P(\F{t})&=\int_{\F{t}} v, \quad \forall t<0,\quad P(\F{t})&=-\int_{\F{t}} v.
    \end{align}

  \item[(iii)] The level sets of $f$ satisfy
    \begin{align}
      \label{eq:setgeompbplus}\forall t>0,\quad \forall G\subset \RR^2, \abs{G}<+\infty,\quad P(G) - \int_{G} v &\geq P(\F{t})-\int_{\F{t}} v, \\
      \label{eq:setgeompbminus}\forall t<0,\quad \forall G\subset \RR^2, \abs{G}<+\infty,\quad P(G) + \int_{G} v &\geq P(\F{t})+\int_{\F{t}} v. 
    \end{align}
  \end{itemize}
   \end{prop}

\begin{proof}
  $(iii) \Rightarrow (i)$ It suffices to use the coarea formula $\int \abs{Dg} = \int_{-\infty}^{\infty} P(\G{t})\d t$ and Fubini's theorem in \begin{equation}
    \int_{\RR^2} gv = \int_{\RR^2}\left( \int_0^{+\infty}\bun{g(x)\geq t}v(x)\d t -  \int_{-\infty}^0\bun{g(x)\leq t}v(x)\d t \right)\d x,
  \end{equation}
and similarly for the level sets of $f$.

$(i) \Rightarrow (ii)$ Using~\eqref{eq:tvsubdiffu}, we see that $v\in \partial J(0)$ and $\int_{\RR^2}fv=J(f)$. 
From $v\in \partial J(0)$, and choosing $\pm\bun{F}$ (for any $F\subset \RR^2$ with $\abs{F}<+\infty$) in the subdifferential inequality, we infer that  $P(F)\pm\int_{\RR^2}\bun{F}v \geq 0$. Now, $\int_{\RR^2}fv=J(f)$ rewrites
\begin{align*}
  0 =\int_0^{+\infty}\left( P(\F{t})-\int_{\F{t}}v\right)\d t + \int_{-\infty}^0\left(P(\F{t})+\int_{\F{t}}v \right)\d t.
\end{align*}
Since the integrands are nonnegative, we obtain that for a.e. $t\in \RR$, $P(\F{t})=\sign(t)\int_{\F{t}}v$. In fact, the equality holds for all $t\neq 0$. Indeed, for $t>0$, we may find a sequence $t_n\nearrow t$ as $n\to +\infty$ such that $ P(\F{t_n})=\int_{\F{t_n}}v$. Since $\abs{\F{t}}<+\infty$ and by monotonicity, $\bun{\F{t_n}}$ converges in $\LDD$ towards $\bun{\F{t}}$ and we have
\begin{equation*}
  P(\F{t})=P\left(\bigcap_{n\in\NN} \F{t_n}\right)\leq \liminf_{n\to +\infty} P\left(\F{t_n}\right)= \liminf_{n\to +\infty} \int_{\F{t_n}}v = \int_{\F{t}} v.
\end{equation*}
The converse inequality holds from the fact that $v\in \partial J(0)$ so that $\int_{\F{t}} v\leq   P(\F{t})$.
In a similar way, we may prove that for all $t<0$, $P(\F{t})= \int_{\F{t}} v$.

$(ii)\Rightarrow (iii)$ From $v\in \partial J(0)$, we infer that $P(G) \pm \int_{G} v\geq 0$ for any $G\subset \RR^2$ with $\abs{G}<+\infty$. Since $P(\F{t})-\sign(t)\int_{\F{t}}v=0$, we obtain the claimed result.
\end{proof}

As a consequence of Proposition~\ref{prop:prelim-subdiff}, if we are given $v\in \partial J(f)$ rather than $f$, we may control the localization of the support of $Df$ simply by studying the solutions of~\eqref{eq:prelim-setgeomcalib}. The following proposition formalizes this idea.

\begin{prop}\label{prop:prelim-suppdf}
  Let $f\in \LDD$ with $J(f)<+\infty$,  $v\in \partial J(f)$ and let $\supp (Df)$ denote the support of the Radon measure $Df$. Then
\begin{align}
  \label{eq:prelim-suppdfequal}  \supp(Df)&=\overline{\bigcup\enscond{\partial^*\F{t}}{t\in \RR\setminus\{0\}}}\\
 \label{eq:prelim-suppdfinclude}&\subseteq \overline{\bigcup\enscond{\partial^*F}{\abs{F}<+\infty \qandq P(F)=\pm \int_F v }}. 
\end{align}
\end{prop}

\begin{proof}
  Let $x\in \RR^2\setminus\supp (Df)$. There exists $r>0$ such that $\abs{Df}(B(x,r))=0$, hence $f$ is constant in $B(x,r)$, identically equal to some $C\in \RR$. Depending the value of $t\in \RR$, we see that either $B(x,r)\subseteq \F{t}$ or $B(x,r)\cap \F{t}=\emptyset$. In any case, $\partial^*\F{t}\subseteq \RR^2\setminus B(x,r)$. As a result $x\in \RR^2\setminus \overline{\bigcup\enscond{\partial^*\F{t}}{t\in \RR}}$, which proves that $\overline{\bigcup\enscond{\partial^*\F{t}}{t\in \RR}}\subseteq  \supp(Df)$. 

  For the converse inclusion, let $x\in \supp(Df)$, so that for all $r>0$, $$\abs{Df}(B(x,r))>0.$$ We apply the coarea formula 
\begin{equation*}
  \abs{Df}(B(x,r))= \int_{-\infty}^{\infty} P(\F{t},B(x,r))\d t,
\end{equation*}
to see that $\HDU{\partial^*\F{t}\cap B(x,r)}>0$ for some $t\in \RR$, hence $$B(x,r)\cap \bigcup\enscond{\partial^*\F{t}}{t\in \RR}\neq \emptyset.$$ Since this is true for all $r>0$, we see that $x\in \overline{\bigcup\enscond{\partial^*\F{t}}{t\in \RR}}$.

Now, we prove the last inclusion in~\eqref{eq:prelim-suppdfinclude}. First, we observe that 
\begin{equation*}
  \overline{\bigcup\enscond{\partial^*\F{t}}{t\in \RR}} = \overline{\bigcup\enscond{\partial^*\F{t}}{t\neq 0}}\cup \overline{\partial^*\F{0}}.
\end{equation*}
By Proposition~\ref{prop:prelim-subdiff}, we know that for every $t\neq 0$, $\F{t}$ satisfies $\abs{\F{t}}<+\infty$ and $\pm\int_{\F{t}}v=P(\F{t})$. Hence
\begin{align*}
  \overline{\bigcup\enscond{\partial^*\F{t}}{t\neq 0}}&\subseteq \overline{\bigcup\enscond{\partial^*F}{\abs{F}<+\infty \qandq P(F)=\pm \int_F v }},\\
  \intertext{and it is sufficient to prove that}
  \overline{\partial^*\F{0}}&\subseteq \overline{\bigcup\enscond{\partial^*\F{t}}{t\neq 0}}.
\end{align*}

Let $x\in \partial^*\F{0} = \partial^*\left(\RR^N\setminus \bigcup_{k\in\NN^*}\F{-1/k} \right) =  \partial^*\left(\bigcup_{k\in\NN^*}\F{-1/k} \right)$.
Then for all $r>0$, 
\begin{equation*}
  \abs{B(x,r)\cap \bigcup_{k\in\NN^*}\F{-1/k}}>0 \qandq  \abs{B(x,r)\setminus \bigcup_{k\in\NN^*}\F{-1/k}}>0
\end{equation*}
In particular, there exists $k_0$ such that $\abs{B(x,r)\cap \F{-1/k_0}}>0$, and moreover $\abs{B(x,r)\setminus \F{-1/k_0}}\geq  \abs{B(x,r)\setminus \bigcup_{k\in\NN^*}\F{-1/k}}>0$.
Hence $\bun{\F{-1/k_0}}$ is not constant in $B(x,r)$, so that $\partial^* \F{-1/k_0}\cap B(x,r)\neq \emptyset$.
As a result, $B(x,r)\cap \bigcup\enscond{\partial^*\F{t}}{t\neq 0}\neq \emptyset$ for all $r>0$, which proves that $x\in \overline{\bigcup\enscond{\partial^*\F{t}}{t\neq 0}}$.

\end{proof}

\subsection{The prescribed mean curvature problem}
As a consequence of Propositions~\ref{prop:prelim-subdiff} and~\ref{prop:prelim-suppdf}, we are led to study the solutions of the \textit{prescribed curvature problem}
\begin{equation}\label{eq:prelim-varcurv}
  \min_{\substack{X\subset \RR^N\\\abs{X}<+\infty}} P(X) + \int_X H
\end{equation}
for $H=\pm v$, where $v\in \partial J(f)$ is fixed.
Following~\cite{Barozzi}, if $E\subset \RR^2$ is a solution to~\eqref{eq:prelim-varcurv}, we say that $v$ is a \textit{variational mean curvature}~\footnote{The careful reader will note that we make a slight abuse in the terminology since in~\cite{Barozzi, massari1994variational} the function $H$ is assumed to be integrable, and the condition $\abs{F}<+\infty$ is not imposed. We make this slight abuse since the local properties of the sets studied in~\cite{Barozzi, massari1994variational} also hold for the solutions of~\eqref{eq:prelim-varcurv}.} for $E$. Depending on the integrability of $H$, the solutions of such a problem have the following regularity properties. 
\begin{prop}[\cite{Ambrocorso}]\label{prop:regu_ambro}
Assume that $H\in L^p_{loc}(\RR^N)$ for some $p\in (N,+\infty]$, and let $E\subseteq \RR^N$ be a nonempty solution of~\eqref{eq:prelim-varcurv}. Then $\Sigma=\partial E\setminus \partial^*E$ is a closed set of Hausdorff dimension at most $N-8$, and $\partial^* E$ is a $\Cder{1,\alpha}$ hypersurface for all $\alpha<(p-N)/2p$.

 If $p=\infty$, then $\partial^* E$ is $\Cder{1,\alpha}$ for all $\alpha>0$, and if additionally $N=2$, then $\partial^* E$ is $\Cder{1,1}$.
\end{prop}
Let us comment on the term variational curvature. Let $x\in \partial^* E$. Up to a translation and rotation we may assume that $\partial^* F$ coincides locally with the graph of some $\Cder{1,\alpha}$ function $\psi: B(0,r)\rightarrow (-r,r)$ such that $\nabla \psi(0)=0$.
If $H$ is continuous in an open $A$, then it is possible to prove~\cite[Th. 1.1.3]{Ambrocorso} that the ``mean'' curvature is equal to $-\frac{1}{N-1} H$,
\begin{equation*}
 \frac{1}{N-1}\divx\left(\frac{\nabla \psi(z)}{\sqrt{1+\abs{\nabla \psi(z)}^2}}  \right) = \frac{1}{N-1} H\left((z,\psi(z))\right),
\end{equation*}
in the sense of distributions. If $N=2$, this equation holds in the classical sense and $\partial^*E\cap A$ is in fact $\Cder{2}$.

The integrability $p$ of $H$ is crucial. For instance, if $p=1$, it implies nothing on the regularity of $E$ since every set of finite perimeter has a variational mean curvature in $L^1$~\cite{Barozzi}. The case $p=N$ which we are interested in is a limit case, and counterexamples in~\cite{massari1994variational,elisabetta2013sets} are provided where  the Hausdorff dimension of $\partial E\setminus \partial^*E$ is more than $N-8$.

However, we may rely on the weak regularity theorem~\cite[Th.~3.6]{massari1994variational} (see also \cite{gonzales1993boundaries}) which ensures that for all $x\in \partial F$,\begin{equation}\label{eq:weak_regularity}
  1>  D_F(x)\eqdef \lim_{r\to 0^+} \frac{\abs{F\cap B(x,r)}}{\abs{B(x,r)}}>0.
\end{equation}
Furthermore, in the case of $N<8$, we have that $D_F(x) = 1/2$.

In particular, the topological boundary $\partial F$ is equal to the \textit{essential boundary} $\partial^M F$,
\begin{align*}
  \partial F &= \partial^M F \eqdef \enscond{x\in \RR^2}{\overline{D_F}(x)>0 \qandq \underline{D_F}(x)>0},\\
  \qwhereq \overline{D_F}(x)&=\limsup_{r\to 0^+} \frac{\abs{F\cap B(x,r)}}{\abs{B(x,r)}} \qandq \underline{D_F}(x)=\liminf_{r\to 0^+} \frac{\abs{F\cap B(x,r)}}{\abs{B(x,r)}}.
\end{align*}

Furthermore, it was shown in \cite{Paolini} that  if $H\in L^N(\RR^N)$, then $\partial^* E$ is a $C^{0,\alpha}$ hypersurface up to some possible singularities.
Thus, in the case where $p=N$, although Proposition \ref{prop:regu_ambro} cannot be applied, the boundary $\partial F$ does not contain wild singular points such as cusps or points of zero density. In Section \ref{sec:props_weak_reg}, we apply \eqref{eq:weak_regularity}, observing that this weak regularity holds uniformly for the boundaries of the level sets of solutions to ($\Pp_\la(f+w)$) in some low noise regime.

\subsection{Decomposition of boundaries into Jordan curves}\label{sec:jordan_decomp}
We shall occasionally rely on the results on the decomposition of sets with finite perimeter provided in~\cite{ambcasmas99}.

 Let $E$ be a set of finite perimeter. By~\cite[Corollary 1]{ambcasmas99}, $E$ can be decomposed into an at most countable union of its $M$-connected components
\begin{equation*}
  E =\bigcup_{i\in I \subseteq \NN} E_{i} \qwhereq P(E)=\sum_{i\in I} P(E_{i}), \quad \abs{E_{i}}>0.
\end{equation*}
and each $M$-connected component can be decomposed as 
$$
E_{i}=\interop(J_i^+)\setminus \bigcup_{j\in L_i}\interop(J_j^-)\qandq \partial^M E_{i} = J_i^+ \cup \bigcup_{j\in L_i} J_j^- \pmod{\Hh^{1}},
$$ where each $J_k^{\pm}$ is a \textit{rectifiable Jordan curve}, $L_i\eqdef\enscond{j\in \NN}{\interop(J_j^-)\subseteq \interop(J_i^+)}$. Here $\interop(J)$ denotes the interior of a Jordan curve (but when the context is clear, we shall also use $\interop$ to denote the topological interior).

Moreover,
\begin{equation*}
  P(E)=\sum_i \HDU{J_i^+} +\sum_j \HDU{J_j^-} \qandq P(\interop J_i^\pm)= \HDU{J_i^\pm} \mbox{ for all } i.
\end{equation*}

\begin{rem}\label{rem:decomp}
  Let $E\subset \RR^2$ with $\abs{E}<+\infty$ such that $P(E)=\int_E v$, where $v\in \partial J(0)$. Let us decompose $E$ into its $M$-connected components, $E =\bigcup_{i\in I} E_{i}$, where we can assume that $I$ is either $\NN$ or of the form $\{0, 1,\ldots, n\}$. We observe that $\{E_i\}_{i\in I, i\geq 1}$ yields the decomposition of $E\setminus E_0$ into its M-connected components. Hence,
\begin{align*}
  0&=  P(E)- \int_E v \\
   &= P(E_{0}) -\int_{E_{0}}v + P\left(\bigcup_{i\in I, i\geq 1} E_{i}\right)-\int_{\bigcup_{i\in I, i\geq 1} E_{i}}v.
\end{align*}
Since $P(E_{0}) -\int_{E_{0}}v\geq 0$ and $ P\left(\bigcup_{i\in I, i\geq 1} E_{i}\right)-\int_{\bigcup_{i\in I, i\geq 1} E_{i}}v\geq 0$, we deduce that those inequalities are in fact equalities. By induction, we deduce that for all $i\in I$,
\begin{equation*}
P(E_{i}) -\int_{E_{i}}v =0.
\end{equation*}
Now decomposing, $\partial^ME_{i}$ into rectifiable Jordan curves, this equivalent to 
\begin{align*}
0  &= \left(P(\interop J_i^+)-\int_{\interop J_i^+}v\right)  +\sum_{j\in L_i} \left(P(\interop J_j^-)+\int_{\interop J_j^-}v\right) 
\end{align*}
Since each Jordan curve $J$ satisfies $P(\interop J)-\abs{\int_{J}v}\geq 0$, we see that for all $i$ and $j$ in the decomposition,
\begin{equation*}
  P(\interop J_i^+)=\int_{J_i^+}v \qandq P(\interop J_j^-)=-\int_{J_j^-}v.
\end{equation*}
Similarly, we may prove that if $P(E)=-\int_E v$,
\begin{align*}
  P(E_i)&=-\int_{E_i}v,\\
  P(\interop J_i^+)&=-\int_{J_i^+}v \qandq P(\interop J_j^-)=\int_{J_j^-}v.
\end{align*}
\end{rem}


\section{Duality for the study of the low noise regime}\label{sec:duality}

\subsection{Dual problems and ``dual certificates''}

We are interested in solving:
\begin{equation}
\min_{u\in \LDD} J(u) + \frac{1}{2\la}\normLdeux{f-u}^2.\tag{$\Pp_\la(f)$}\label{eq-rof-primal}\end{equation}
where $J(u)=\int_{\RR^2} |Du|\in \RR_+\cup\{+\infty\}$.

Using the framework and notations of~\cite{EkelandTemam}, we set $V=\LDD$, $\La=Id$, $Y=\LDD$, $F=J$, $G=\frac{1}{2\la}\|\cdot-f \|^2$ and we compute the Fenchel-Rockafellar dual problem as
\begin{align}
  \sup_{v\in \partial J(0)} \langle f,v\rangle -\frac{1}{2\la} \normLdeux{v}^2, \tag{$\Dd'_\la(f)$}\label{eq-rof-dualsup}\\
  \mbox{or equivalently }\quad \inf_{v\in \partial J(0)} \normLdeux{\frac{f}{\la} - v}^2 \tag{$\Dd_\la(f)$}\label{eq-rof-dual}
\end{align}
It is easy to check that Problem~\eqref{eq-rof-primal} is stable in the sense of~\cite{EkelandTemam}. In particular, there exists a solution to\eqref{eq-rof-dualsup} and strong duality holds between \eqref{eq-rof-primal} and \eqref{eq-rof-dualsup}, namely $\inf\eqref{eq-rof-primal} = \sup \eqref{eq-rof-dualsup}$. In fact~\eqref{eq-rof-dual} is a projection problem onto a nonempty closed convex set, hence it always has a unique solution. 

Observe that formally, the limit of~\eqref{eq-rof} as $\la\to 0^+$ is the trivial problem
\begin{equation}
\min_{u\in \LDD} J(u) \quad \mbox{s.t.}\quad  u=f,
\tag{$\Pp_0(f)$}\label{eq:rof-noiseless}
\end{equation}
having $u=f$ as solution.
The dual associated with this ``exact reconstruction problem'' is
\begin{align}
  \sup_{v\in \partial J(0)} \langle f,v\rangle, \tag{$\Dd_0(f)$}\label{eq-constr-dual}
\end{align}
having $\partial J(f)$ solutions.
Here again, strong duality holds, since it is possible to prove that~\eqref{eq-constr-dual} is stable. However, a solution to~\eqref{eq-constr-dual} does not always exist since it may be that $\partial J(f)=\emptyset$.

The main point in studying the dual problems is that their solutions $v_\la$ are related to the primal solutions $u_\la$ by the extremality relations
\begin{align*}
  v_\la &\in \partial J(u_\la)\\
  v_\la &= \frac{1}{\la}(f-u_\la),
\end{align*}
which enables to study the support of $Du_\la$ (see Section~\ref{sec:prelim}).
For the noiseless problem, the extremality relation is $v\in \partial J(f)$, for every $v$ solution to~\eqref{eq-constr-dual}.
The term ``certificate'' stems from the fact that if $u\in\LDD$ and $v\in\LDD$ satisfy the extremality relations, then $u$ is a solution of the primal problem and $v$ is a solution of the dual problem.

\subsection{Low noise regimes and the minimal norm certificate}
We shall often consider noisy observations $f+w$, where $w\in \LDD$, and from now on we denote by $\ulaw$ (resp. $\vlaw$) the unique solution to~$\Pp_\la(f+w)$ (resp.~$\Dd_\la(f+w)$).

Given $\la_0>0$, $\alpha_0>0$, we consider the low noise regime
\begin{equation}\label{eq:intro-lnr}
  \lnr{\la_0}{\alpha_0}\eqdef \enscond{(\la,w)\in \RR_+\times \LDD}{0\leq\la \leq \la_0\qandq\normLdeux{w}\leq \alpha_0 \la }.
\end{equation}

The dual solution $\vlaw$ being the projection of $(f+w)/\la$ onto a convex set, the non-expansiveness of the projection yieds
\begin{equation*}
  \forall (\la,w)\in \RR^*_+\times\LDD, \quad \normLdeux{\vlao-\vlaw}\leq \frac{\normLdeux{w}}{\la}\leq \alpha_0. 
\end{equation*}
As a result, the properties of $\vlaw$ are governed by those of $\vlao$, and it turns out that the properties of $\vlao$ are governed, in the low noise regime, by those of a specific solution to~\eqref{eq-constr-dual}, as the next result hints. The proof is identical to the one in \cite{duvalpeyre13}.
\begin{prop}
  Let $f\in \LDD$, $J(f)<+\infty$, and assume that $\partial J(f)\neq \emptyset$. Let $\voo\in \LDD$ be the solution to \eqref{eq-constr-dual} with minimal $L^2$ norm. Then
   \begin{align*}
    \lim_{\la\to 0^+} \vlao = \voo \quad \mbox{ strongly in } \LDD,
  \end{align*}
\end{prop}
We call $\voo$ the \textbf{minimal norm certificate} for $f$. It is also known as the \textit{minimal section} in maximal monotone operator theory. The goal of the present paper is to show that $\voo$ governs the support of the solutions in the low noise regime. In particular, $\voo$ determines whether the support of $D\ulaw$ is close to the support of $Df$ in that regime.

In the next paragraphs, we illustrate the minimal norm certificate in simple cases.

\subsection{The minimal norm certificate for calibrable sets}
\begin{prop}[Minimal norm certificates for calibrable sets]\label{prop:mincertcalib}
  Let $C\subseteq \RR^2$ be a bounded calibrable set and $f=\bun{C}$. Then the minimal norm certificate is $\voo=h_C\bun{C}$, where $h_C=\frac{P(C)}{\abs{C}}$.
  \label{prop-mincertif-calib}
\end{prop}
We provide two different proofs of the above result, each highlighting different aspects of the minimal norm certificate.

\begin{proof}[Proof ($\voo$ as a limit)]
From \cite{beltvflow02}, we know that for a calibrable set $C\subseteq \RR^2$, the solution to \eqref{eq-rof-primal} with $f=\bun{C}$ is given by $\ulao = (1 -\la h_C)_+ \bun{C}$.
From the optimality conditions,
\begin{align*}
    \vlao &= \frac{1}{\la}(f-\ulao),
\end{align*}
we obtain that $\vlao =h_C\bun{C}$ provided $0< \la \leq \frac{1}{h_C}$. Taking the limit as $\la \to 0^+$, we obtain $\voo=h_C\bun{C}$.
\end{proof}
\begin{proof}[Another Proof ($\voo$ as a minimal norm element)]
  Observe that for all $f\in \LDD$ with $J(f)<+\infty$, and $v\in \LDD$, $v$ is a solution to~\eqref{eq-constr-dual} if and only if $v\in \partial J(f)$. For $C\subset \RR^2$ bounded calibrable, we obtain that $h_C \bun{C}$ is a solution to \eqref{eq-constr-dual}. It remains to prove that it is the one with minimal norm. 
   
Let $v\in \LDD$ be any solution to \eqref{eq-constr-dual}. By the Cauchy-Schwarz inequality
\begin{align*}
  \sup \Dd_0(\bun{C}) = \langle v,\bun{C}\rangle \leq \normLdeux{v}\normLdeux{\bun{C}}= \sqrt{|C|} \normLdeux{v}.
\end{align*}
But $\normLdeux{h_C\bun{C}}=\frac{P(C)}{\sqrt{|C|}}= \frac{\sup \Dd_0(\bun{C})}{\sqrt{|C|}}$, so that $h_C\bun{C}$ has minimal norm.
\end{proof}

\subsection{The minimal norm certificate for smooth convex sets}\label{sec:MNC_convex}
Let $C$ be a nonempty open bounded convex subset of $\RR^2$. Given $\rho>0$ we denote by $C_\rho$ the opening of $C$ by open balls with radius $\rho$, namely $C_\rho=\bigcup_{B(y,\rho)\subseteq C} B(y,\rho)$.
 For $f=\bun{C}$, it is proved in~\cite{altercalib05,alterconvex05,Chambolle10anintroduction} that the solution $\ulao$ to \eqref{eq-rof} is 
$$
\ulao = \left(1+ \la v_C \right)^+ \bun{C},
$$
where, by letting $R$ be such that $C_R$ is the maximal calibrable set in $C$, the function $v_C:\RR^2\to \RR$ is defined by
\begin{equation}\label{eq:convexvc}
v_C(x) \eqdef \begin{cases}
1/R & x\in C_{R}\\
1/r & x\in \partial C_r, ~ r\in [0,R) \\
0 & \text{otherwise.}
\end{cases}
\end{equation}
Since $\vlao = \la^{-1}(f- u_{\la,0})$, it follows that
\begin{equation}\label{eq:convexcertif}
\vlao(x) = \begin{cases}
v_C(x) & x\in C_\la \\
1/\la & x\in C\setminus C_\la\\
0& \text{otherwise.}
\end{cases}
\end{equation}

Now we assume that $C\subset \RR^2$ has $C^{1,1}$ boundary, and we let $\rho_0>0$ such that
\begin{align*}
  \kappa_{\partial C}(x)\leq \frac{1}{\rho_0} \quad \mbox{ for } \Hh^1\mbox{-a.e. }x\in \partial C,
\end{align*}
where $\kappa_{\partial C}$ is the curvature of $\partial C$ (defined $\Hh^1$-almost everywhere on $\partial C$).
We shall need the following lemma.
\begin{lem}[\cite{beltvflow02}]
  Let $C\subset \RR^2$ be a bounded open convex set. The following conditions are equivalent:
\begin{itemize}
  \item there exists $\rho >0$ such that $C=C_\rho$;
  \item $\partial C$ is of class $C^{1,1}$ and $\ess \sup_{p\in \partial C}\kappa_{\partial C}(p) \leq \frac{1}{\rho}$.
\end{itemize}
\end{lem}
Since for $0<r\leq \rho_0$, $C_{\rho_0}\subseteq C_r\subseteq C$, we see that  $C_r=C$ for $0<r\leq \rho_0$.

As a result, $\la \mapsto \vlao$ is constant on $(0,\rho_0]$, and the minimal norm certificate is thus
\begin{equation}\label{eq:convexmincertif}
  \voo = v_C.
\end{equation}
It turns out that $\voo$ is precisely the subgradient constructed by Alter \textit{et al.}~\cite{alterconvex05} for the evolution of convex sets by the total variation flow. It is instructive to look at the associated vector field $z_0$ such that $\divx z_0=\voo$

For every $x\in \interop(C)\setminus\overline{C_R}$, there exists a unique $r(x)$ such that $x\in \partial C_{r(x)}$, and $x$ belongs to an arc of circle of radius $r(x)$. Defining $\nu(x)$ as the outer unit normal to this set, 
 define 
\begin{align*}
  z_0(x)\eqdef\left\{ \begin{array}{l}\nu(x) \mbox{ if }x\in \interop(C)\setminus C_R\\ 
      z_{C_R}(x) \mbox{ if } x\in C_R\\
    \overline{z}(x) \mbox{ if } x\in \RR^2\setminus C. 
  \end{array} \right.
\end{align*}
where $\overline{z}$ is a calibration of $\RR^2\setminus C$ (see Section~\ref{sec:calibout}). As for $z_{C_R}$, since $C_R$ is calibrable ($C_R$ is then the Cheeger set of $C$) there exists a vector field $z_{C_R}$ such that $|z_{C_R}|\leq 1$, $\theta(z_{C_R},-D\bun{C_R})=1$, and $\divx z_{C_R}=h_{C_R} \bun{C_R}$ with $h_{C_R}=P(C_R)/\abs{C_R}$. 

It is proved in~\cite{alterconvex05} that $\divx z_0=v_C$ (in the sense of distributions).

%
%

It is notable that the construction is quite similar to the one proposed by Barrozzi et al. in \cite{Barozzi} and studied in \cite{massari1994variational}. In particular, the $L^2$-minimality (or even $L^p$ minimality) of the above constructions is already noted in  \cite{massari1994variational}.

\section{Properties of the level sets in the low noise regime}\label{sec:proplev}

In this section, we rely on the properties of the minimal norm certificate $\voo$ to study the solutions of~\eqref{eq:prelim-suppdfinclude} for $v=\vlaw$ in a low noise regime. More precisely we study the elements of
\begin{align}\label{eq:proplev-eq}
  \FFlaw\eqdef\enscond{E\subset \RR^2}{\abs{E}<+\infty, \qandq  \pm\int_E \vlaw = P(E)},
\end{align}
for $(\la, w)\in\lnr{\la_0}{\alpha_0}$ with $\la_0>0$, $\alpha_0>0$ small enough.
In the following, we denote by $\Elaws$ or $\Elaw$ any nonempty element of $\FFlaw$. Let us emphasize that we allow the case $(\la,w)=(0,0)$, in which case $\vlaw$ in~\eqref{eq:proplev-eq} is the minimal norm certificate $\voo$. Typically, from Section~\ref{sec:prelim}, one may think of $\Elaw$ as a level set of $\ulaw$ (or $f$, for $(\la,w)=(0,0)$), but additional  sets may solve~\eqref{eq:proplev-eq}.

\subsection{Upper and lower bounds}
In the following lemmas, we prove that there exist uniform upper and lower bounds on the perimeters and the measures of all sets in $\FFlaw$ with $(\la,w)\in D_{1,\sqrt{\cD}/4}$.

\begin{lem}\label{lem:unif_bd_lev_sets}
  Let $\alpha_0\leq \frac{\sqrt{\cD}}{4}$, where $\cD=4\pi$ is the isoperimetric constant. 
Then,
\begin{align}
  \label{eq:proplev-persup}  \sup \enscond{P(\Elaws)}{\Elaws\in \FFlaw, \qandq (\la,w)\in \lnr{1}{\alpha_0} }&<+\infty,\\
  \label{eq:proplev-areasup} \sup \enscond{\abs{\Elaws}}{\Elaws\in \FFlaw, \qandq (\la,w)\in \lnr{1}{\alpha_0} }&<+\infty.
\end{align}
\end{lem}

\begin{proof}
  First, we prove~\eqref{eq:proplev-persup}. 
Since $\lim_{\la\to 0^+} \vlao=\voo$ in $\LDD$, the mapping $\la\mapsto \vlao$ is continuous on the compact set $[0,1]$, hence bounded. Moreover, the family $\{\vlao\}_{0\leq \la \leq 1}$ is $L^2$-equiintegrable so that given any $\varepsilon>0$, there exists $R>0$ such that $\int_{\RR^2\setminus B(0,R)}\vlao^2\leq \varepsilon^2$ for all $\la \in [0,1]$.
Let us also assume that $\alpha_0\leq \varepsilon$ (so that $\normLdeux{\vlaw-\vlao}\leq\frac{\normLdeux{w}}{\la}\leq \varepsilon$).

To simplify the notation, we denote by $\Elaws$ (rather than $\Elaw$) any nonempty set such that $P(\Elaws)=\pm\int_{\Elaws}\vlaw$.

Now, the triangle and the Cauchy-Schwarz inequalities yield
  \begin{align*}
    P(\Elaws)&\leq \abs{\int_{\Elaws}(\vlaw-\vlao)} + \abs{\int_{\Elaws}\vlao}\\
            &\leq\varepsilon\sqrt{\abs{\Elaws}}+ \abs{\int_{\Elaws\cap B(0,R)}\vlao} + \abs{\int_{\Elaws\setminus B(0,R)}\vlao}\\
            &\leq \varepsilon\sqrt{\abs{\Elaws}}+ \sqrt{\abs{B(0,R)}}\normLdeux{\vlao} + \sqrt{\abs{\Elaws\setminus B(0,R)}}\sqrt{\int_{\RR^2\setminus B(0,R)}\vlaw^2}\\
            &\leq \left(\varepsilon +\sup_{\la\in[0,1]}\normLdeux{\vlao}\right) \sqrt{\abs{B(0,R)}}+ 2\varepsilon\sqrt{\abs{\Elaws\setminus B(0,R)}}\\
  \end{align*}
%
  Recalling that $P(\Elaws\setminus B(0,R))\leq P(\Elaws)+ P(B(0,R))$, and using the isoperimetric inequality, we obtain
\begin{align*}
  \sqrt{\abs{\Elaws\setminus B(0,R)}}\leq  \frac{1}{\sqrt{\cD}} \left(P(\Elaws) + P(B(0,R))\right),
\end{align*}
 where is $\cD$ the isoperimetric constant. We choose $\varepsilon=\frac{\sqrt{\cD}}{4}$ and we define $C=\left(\varepsilon +\sup_{\la\in[0,1]}\normLdeux{\vlao}\right) \sqrt{\abs{B(0,R)}}$ so as to get 
\begin{align*}
  \label{eq:}
  P(\Elaws)-\frac{1}{2}\left(P(\Elaws) + P(B(0,R))\right)\leq C.
\end{align*}
We obtain that $P(\Elaws)$ is uniformly bounded in $\Elaws\in \FFlaw$, $(\la,w)\in \lnr{1}{\alpha_0}$.

As for~\eqref{eq:proplev-areasup}, the isoperimetric inequality yields
\begin{equation*}
  \abs{\Elaws}\leq \frac{1}{\cD} (P(\Elaws))^{2}
\end{equation*}
hence $\abs{\Elaws}$ is uniformly bounded in $\Elaws\in \FFlaw$, $(\la,w)\in \lnr{1}{\alpha_0}$.
\end{proof}

Conversely, the perimeters and areas of the solutions are also lower bounded, as the next result shows.
  
\begin{lem}\label{lem:lower_bound_lev_sets}
  Let $\alpha_0\leq \frac{\sqrt{\cD}}{4}=\sqrt{\pi}/2$. Then, 
\begin{align}
  \label{eq:proplev-perinf}  \inf \enscond{P(\Elaws)}{\Elaws\in \FFlaw, \Elaws\neq \emptyset \qandq (\la,w)\in \lnr{1}{\alpha_0} }&>0,\\
  \label{eq:proplev-areainf} \inf \enscond{\abs{\Elaws}}{\Elaws\in \FFlaw, \Elaws\neq \emptyset  \qandq (\la,w)\in \lnr{1}{\alpha_0} }&>0.
\end{align}
Moreover, there exists a number $N_0\in\NN$ such that the number of $M$-connected components $\Elaws$ and the number of Jordan curves in the essential boundary $\partial^M\Elaws$ is uniformly bounded by $N_0$ for all $\Elaws\in \FFlaw$, $(\la,w)\in \lnr{1}{\alpha_0}$.
\end{lem}
\begin{proof}
  By the $L^2$-equiintegrability of the family $\{\vlao\}_{0\leq \la \leq 1}$, for all $\varepsilon>0$, there exists $\delta$ such that for all $\Elaws\subset \RR^2$, with $\abs{\Elaws}\leq \delta$,
\begin{equation*}
  \int_{\Elaws} \vlao^2 \leq \varepsilon^2.
\end{equation*}

We choose $\varepsilon = \frac{\sqrt{\cD}}{4}$, $0<\alpha_0\leq \varepsilon$, and we consider by contradiction a set $\Elaws\in \FFlaw$ such that $0<\abs{\Elaws}\leq \delta$.
Then,
\begin{align*}
  P(\Elaws)  &\leq \abs{\int_{\Elaws}(\vlaw-\vlao)} + \abs{\int_{\Elaws}\vlao}\\
             &\leq \normLdeux{\vlaw-\vlao}\sqrt{\abs{\Elaws}} + \sqrt{\int_{\Elaws} \vlao^2}\sqrt{\abs{\Elaws}}\\
             &\leq 2\varepsilon \sqrt{\abs{\Elaws}}\leq \frac{1}{2} P(\Elaws),
\end{align*}
by the isoperimetric inequality. Dividing by $P(\Elaws)>0$ yields a contradiction, hence $\abs{\Elaws}> \delta$ for all $\Elaws\neq \emptyset$, that is~\eqref{eq:proplev-areainf}.
We deduce the uniform lower bound on the perimeter \eqref{eq:proplev-perinf} by the isoperimetric inequality.

Now, let us decompose the essential boundary of $\Elaws\in \FFlaw$ into \emph{at most countably} many non trivial Jordan curves $\enscond{J^+_i,J^-_j}{i\in I,j\in J, I\subseteq\NN,J\subseteq \NN}$.
By Remark~\ref{rem:decomp} we know that for each $\sigma\in\{-1,1\}$ and $j\in\NN$, $\abs{\int_{\interop{J^\sigma_j}}v_{\la,w}} = \Hh^1(J^\sigma_j)$, that is $(\interop J^\sigma_j)\in\FFlaw$. As a result $\Hh^1(J^\sigma_j)\geq \mu$, where $\mu$ is the infimum defined in~\eqref{eq:proplev-perinf} (in $I$, $J$, we only consider the non-trivial Jordan curves).
Expressing the perimeter of $\Elaws$ in terms of these Jordan curves, we get
$$
C\geq P(\Elaw) = \sum_{i\in I}\Hh^1(J^+_i)+ \sum_{j\in J}\Hh^1(J^-_j)\geq \mu(\card{I}+\card{J}),
$$
where $C$ is the supremum in~\eqref{eq:proplev-persup}. Hence the number of Jordan curves is at most $C/\mu$, and the same holds for the number of M-connected components.
\end{proof}  

Additionally, the next result shows that the level sets are uniformly contained in some large ball.
\begin{lem}\label{lem:largeball}
  Let $\alpha_0\leq \frac{\sqrt{\cD}}{4}=\sqrt{\pi}$. Then, there exists $R>0$ such that
\begin{equation*}
  \forall (\la,w)\in\lnr{1}{\alpha_0},\  \forall \Elaws\in \FFlaw,\quad \Elaws\subset B(0,R).
\end{equation*}
\end{lem}

\begin{proof} 
We begin with the same equiintegrability argument as in Lemma~\ref{lem:unif_bd_lev_sets}, choosing again $\varepsilon=\frac{\sqrt{\cD}}{4}$.
Now, let $\Elaws\in \FFlaw$. By the results of Section~\ref{sec:jordan_decomp}, we may further decompose, up to an $\Hh^1$-negligible set, its essential boundary  $\partial^M \Elaws$  into a countable union of Jordan curves $J$ which satisfy
\begin{equation*}
  \pm\int_{\interop J}\vlaw = P(\interop J) = \HDU{J}
\end{equation*}
Assume by contradiction that $J$ is such that $(\interop J)\cap B(0,R)=\emptyset$. Then by the isoperimetric inequality,
\begin{align*}
  P(\interop J)\leq \sqrt{\int_{\RR^2\setminus B(0,R)}\vlaw^2}\sqrt{\abs{\interop J}}\leq  \frac{2\varepsilon}{\sqrt{\cD}} P(\interop J). 
\end{align*}
Dividing by $P(\interop J)$ yields a contradiction for $\varepsilon=\frac{\sqrt{\cD}}{4}$ if $J$ is not trivial.
Hence $(\interop J)\cap B(0,R)\neq\emptyset$.
But the uniform bound~\eqref{eq:proplev-persup} also holds for $J$, hence there is some $C>0$ (independent from $(\la,w)\in \lnr{1}{\alpha_0}$) such that $\HDU{J}\leq C$. As a result, $\mathrm{diam} (\interop J)\leq C$ so that $(\interop J)\subset B(0,R+C)$, and since this holds for any $J$ which is involved in the decomposition of $\partial^M \Elaws$, it also holds for all $\Elaws\in \FFlaw$, uniformly in~$(\la,w)\in \lnr{1}{\alpha_0}$ . 
\end{proof}

\begin{rem}\label{rem:countunion}
  Let us divide $\FFlaw$ into two classes corresponding respectively to the condition $\int_E \vlaw =P(E)$ and $-\int_E \vlaw =P(E)$ (the empty set being the only element which belongs to both). A consequence of~\eqref{eq:proplev-areasup} is that \textit{each class is stable by finite or countable union or intersection}.
Indeed, if $E$ and $F$ are two elements of $\FFlaw$ such that $\int_E\vlaw=P(E)$ (and similarly for $F$), the submodularity of the perimeter yields
\begin{equation*}
  P(E\cap F)+P(E\cup F)\leq P(E)+P(F) = \int_E \vlaw +\int_F \vlaw = \int_{E\cap F} \vlaw + \int_{E\cup F}\vlaw.
\end{equation*}
Using the subdifferential inequality (on $\vlaw\in\partial J(0)$) we obtain that $P(E\cap F)= \int_{E\cap F}\vlaw$ and $P(E\cup F)= \int_{E\cup F}\vlaw$.
Iterating, we get for finite union or intersection $P(\bigcup_{k=1}^{n}E_k)= \int_{\bigcup_{k=1}^{n}E_k}\vlaw$ and $P(\bigcap_{k=1}^{n}E_k)= \int_{\bigcap_{k=1}^{n}E_k}\vlaw$. The lower semi-continuity of the perimeter together with $\abs{E_1}<+\infty$ yields
\begin{equation*}
  P\left(\bigcap_{k=1}^{\infty}E_k\right)\leq \liminf_{n\to+\infty} P\left(\bigcap_{k=1}^{n}E_k\right)=\lim_{n\to+\infty} \int_{\bigcap_{k=1}^{n}E_k}\vlaw =\int_{\bigcap_{k=1}^{\infty}E_k}\vlaw,
\end{equation*}
and the converse inequality holds by the subdifferential inequality. As for the union, we know from~\eqref{eq:proplev-areasup} that $\abs{\bigcup_{k=1}^{\infty}E_k}=\sup_{n\in\NN} \abs{\bigcup_{k=1}^{n}E_k}<+\infty$, hence
\begin{equation*}
  P\left(\bigcup_{k=1}^{\infty}E_k\right)\leq \liminf_{n\to+\infty} P\left(\bigcup_{k=1}^{n}E_k\right)=\lim_{n\to+\infty} \int_{\bigcup_{k=1}^{n}E_k}\vlaw =\int_{\bigcup_{k=1}^{\infty}E_k}\vlaw,
\end{equation*}
and the opposite inequality also holds, for the same reason as above.
\end{rem}

\subsection{Weak regularity}\label{sec:props_weak_reg}

In this section, we show that \eqref{eq:weak_regularity} holds uniformly on the boundaries of the  sets in $\FFlaw$ with $(\la,w)\in D_{1,\sqrt{\cD}/4}$. The proof of Proposition \ref{prop:weak_reg} is in fact almost identical to the proof of  \cite[Lem.~1.2]{gonzales1993boundaries}, however, it is included for the sake of completeness, and so as to emphasize the uniformity of this estimate with respect to $(\la,w)$.

\begin{prop}\label{prop:weak_reg} 
There exists $r_0>0$ such that for all $r\in (0, r_0]$ and $\Elaw\in\FFlaw$ with $(\la,w)\in D_{1,\sqrt{\cD}/4}$, 
\begin{equation}\label{eq:weak_reg_2D}
\forall x\in  \partial \Elaw,\quad \frac{\abs{B(x,r)\cap \Elaw}}{\abs{B(x,r)}}\geq \frac{1}{16}\qandq \frac{\abs{B(x,r)\setminus \Elaw}}{\abs{B(x,r)}} \geq \frac{1}{16}.
\end{equation}
\end{prop}

\begin{proof}
  We give the proof for $P(\Elaw)=\int_{\Elaw}\vlaw$, the other case being similar.
Since $\{v_{\la,0} \}_{\la\in [0,1]}$ is equiintegrable, there there exists $r_0>0$ such that for all subsets $E\subset \RR^2$ with $\abs{E}\leq \pi r_0^2$,
\begin{equation}\label{eq:equi_int}
\left(\int_{E} \abs{v_{\la,0}}^2\right)^{1/2} \leq \frac{\sqrt{\cD}}{4}.
\end{equation}
First observe that by optimality of $\Elaw$,
\begin{equation}\label{eq:compare}
P(\Elaw) - \int_{\Elaw}\vlaw \leq P(E\setminus B(x,r)) - \int_{\Elaw\setminus B(x,r)} \vlaw.
\end{equation}
For a.e.\ $r\in (0,r_0]$, $\Hh^1(\partial^* \Elaw \cap \partial B(x,r))=0$, so that~\eqref{eq:compare} yields
$$
\Hh^1(\partial^* \Elaw \cap B(x,r)) - \int_{\Elaw\cap B_r} \vlaw \leq \Hh^1(\partial B_r\cap \Elaw).
$$
By adding $\Hh^1(\partial B(x,r)\cap \Elaw)$ to both sides, it follows that
$$
P(\Elaw \cap B(x,r)) - \int_{\Elaw\cap B(x,r)} \vlaw \leq 2\Hh^1(\partial B(x,r)\cap \Elaw).
$$
By the Cauchy-Schwarz inequality, \eqref{eq:equi_int} and since $\norm{\vlaw- v_{\la,0}}_{L^2}\leq \sqrt{c_2}/4$
$$
P(\Elaw \cap B(x,r)) - \frac{\sqrt{\cD}\abs{\Elaw \cap B(x,r)}^{1/2}}{2} \leq 2\Hh^1(\partial B(x,r)\cap \Elaw).
$$
The isoperimetric inequality then implies that
$$
\sqrt{\cD}\abs{\Elaw\cap B(x,r)}^{1/2}\leq 4\Hh^1(\partial B(x,r)\cap \Elaw).
$$
Let $g(r) = \abs{\Elaw\cap B(x,r)}$. Then $g(r)>0$ since $x\in \partial \Elaw$, and for a.e. $r$,
$g'(r) = \Hh^1(\partial B(x,r)\cap \Elaw)$. Therefore, for a.e.\ $r\in (0,r_0]$,
$$
\sqrt{\cD}\leq 8 \frac{\mathrm{d}}{\mathrm{d}r}\sqrt{g(r)}.
$$
By integrating on both sides,
$$
r\sqrt{\cD} \leq 8 \sqrt{g(r)}.
$$
and the first inequality in \eqref{eq:weak_reg_2D} follows by recalling that $c_2=4\pi$.
The proof of $\abs{B(x,r)\setminus \Elaw} \geq \abs{B(x,r)}/16$ is similar: instead of comparing $\Elaw$ with $\Elaw\setminus B(x,r)$ in \eqref{eq:compare}, simply compare $\Elaw$ with $\Elaw\cup B(x,r)$ and proceed as before.

\end{proof}


\section{The extended support}\label{sec:extended}

Let $f\in \LDD$, with $J(f)<+\infty$, such that $\partial J(f)\neq \emptyset$, or equivalently that~\eqref{eq-constr-dual} has a solution (source condition). 
Let $\voo$ be the corresponding minimal norm certificate and let us define the extended support as
\begin{equation*}
  \ext(Df) \eqdef \overline{\bigcup\enscond{  \supp Dg}{\voo\in \partial J(g)}}
\end{equation*}
As we shall see in Section~\ref{sec:stab_extended_spt}, the extended support governs the location of $\Supp (D\ulaw)$ for $(\la,w)$ in some low noise regime.

\subsection{Properties of the extended support}
A first remark in view of Proposition~\ref{prop:prelim-suppdf} is that we may rewrite the extended support as 
\begin{equation}\label{eq:extdfperim}
  \ext(Df) = \overline{\bigcup\enscond{\partial^*E}{\abs{E}<+\infty \qandq  \pm\int_E \voo =P(E)}}
\end{equation}
The first inclusion is clear by Proposition~\ref{prop:prelim-suppdf}. The converse inclusion is obtained by considering, for any $E$ in the right hand-side, the function $g=\bun{E}$, so as to have $g\in L^2$ and $\voo\in \partial J(g)$ (since $\overline{\partial^*E}=\supp(Dg)$).

From the above equalities, we see that \textit{all the properties of Section~\ref{sec:proplev} (lower and upper boundedness of the perimeter, uniform boundedness\ldots) hold for the elements of the right hand-side whose union determines the extended support.}

The rest of the section is devoted to examples of minimal certificates, in the case of indicator function of convex calibrable sets or more general convex sets.

\subsection{Convex Calibrable sets}\label{sec:ext_calib}

Let $C\subset \RR^2$ be a bounded convex calibrable set. We wish to describe the extended support of $f=\bun{C}$. This may be done by looking at a vector field $z$ with divergence  $\voo$ (see Section~\ref{sec:calibrable}), which is more informative, or by the following approach.

By Proposition~\ref{prop:mincertcalib}, we know that the minimal norm certificate associated to $f=\bun{C}$ is $\voo=h_C\bun{C}$, where $h_C=\frac{P(C)}{\abs{C}}$. By~\eqref{eq:extdfperim}, we are thus led to solve
\begin{align}
  \label{eq:extendedcalibmoins}   \inf_{\substack{E\subset\RR^2\\\abs{E}<+\infty}} P(E)-h_C \abs{E\cap C},\\
  \label{eq:extendedcalibplus}  \mbox{and } \inf_{\substack{E\subset\RR^2\\\abs{E}<+\infty}} P(E)+h_C \abs{E\cap C}.
\end{align}
Problem~\eqref{eq:extendedcalibplus} is trivial and its only solution is $\emptyset$, so that we only focus on~\eqref{eq:extendedcalibmoins}. By Proposition~\ref{prop:prelimcalibminimal}, we see that $E=C$ is a minimizer. 
Moreover, since $C$ is convex, for all $E$ with finite perimeter $P(C\cap E)\leq P(E)$ with strict inequality whenever $\abs{E\setminus C}>0$. As a result, any other solution must satisfy $E\subseteq C$. But with this condition, either $E=\emptyset$ or $E$ is a solution to the Cheeger problem
\begin{equation*}
  \min_{E\subseteq C} \frac{P(E)}{\abs{E}}.
\end{equation*}
The uniqueness of the solution to the Cheeger problem inside any convex set is proved in~\cite{giusti78,alteruniq09}, and we already know that $C$ is optimal. As a result, either $E=\emptyset$ or $E=C$, and eventually
\begin{equation*}
  \ext(Df)= \partial C.
\end{equation*}

\subsection{Smooth convex sets}\label{sec:ext_smooth}
Let $C\subset\RR^2$ be a bounded open convex set with $\Cder{1,1}$ boundary. We describe the extended support of $f=\bun{C}$ by considering the minimal norm certificate $\voo$ defined in~\eqref{eq:convexvc}.
We need to study the solutions of 
\begin{align}
  \label{eq:extendedcvxmoins}   \inf_{\substack{E\subset\RR^2\\\abs{E}<+\infty}} P(E)-\int_E \voo,\\
  \label{eq:extendedcvxplus}  \mbox{and } \inf_{\substack{E\subset\RR^2\\\abs{E}<+\infty}} P(E)+\int_E \voo.
\end{align}
Since $\voo\geq 0$, we see that the only solution to~\eqref{eq:extendedcvxplus} is $\emptyset$.
As for~\eqref{eq:extendedcvxmoins}, the same convexity argument as above shows that any solution must be included in $C$. 

Now let $r\in [\rho_0, R]$, where $1/\rho_0\geq \ess\sup_{x\in\partial C}\kappa(x)$ and $C_r$ be the opening of $C$ with radius $r$ as defined  in Section~\ref{sec:MNC_convex}. Denoting by $\nu_{C_r}$ the outer unit normal to $\partial C_r$, we have
\begin{equation*}
  P(C_r)=\int_{\partial C_r}z_0\cdot \nu_{C_r}\d \Hh^1 =\int_{C_r}\divx z_0 =\int_{C_r}\voo,
\end{equation*}
hence $C_r$ is a solution to~\eqref{eq:extendedcvxmoins}, hence  $\ext(Df) \supseteq \overline{\bigcup\enscond{\partial C_r}{\rho_0\leq r\leq R}}$.

Let us prove that there is no solution $E$ such that the reduced boundary $\partial^*E$ intersects $C_R$. 
By Remark~\ref{rem:countunion}, the solutions to~\eqref{eq:extendedcvxmoins} are stable by intersection. If a solution $E$ is such that $E\cap C_R\neq \emptyset$, then $P(E\cap C_R)=\int_{E\cap C_R}\voo = h_{C_R}\abs{E\cap C_R}$ where $h_{C_R}=\frac{P(C_R)}{\abs{C_R}}$ and $E\cap C_R$ is a solution to the Cheeger problem
\begin{equation*}
  \min_{F\subseteq C_R} \frac{P(F)}{\abs{F}}. 
\end{equation*}
By uniqueness of the Cheeger set of $C_R$, we obtain that $E\cap C_R=C_R$.
Eventually, we have proved 
\begin{equation}
  \ext(Df) =\overline{\bigcup\enscond{\partial C_r}{\rho_0\leq r\leq R}}.
\end{equation}

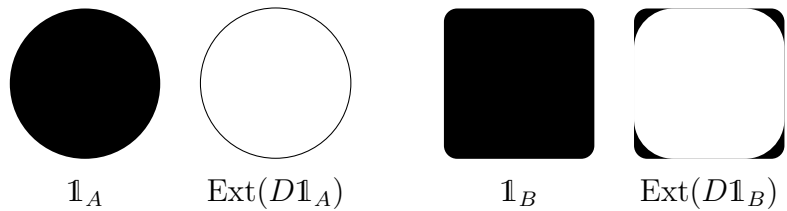
\begin{figure}[H]
\begin{center}
\begin{tabular}{c@{\hspace{15pt}}c@{\hspace{35pt}}c@{\hspace{15pt}}c}
        \begin{tikzpicture}[scale=2]

\fill[black] (.5,.5) circle(5mm);
\end{tikzpicture}
&
\begin{tikzpicture}[scale=2]

\draw[black] (.5,.5) circle(5mm);
\end{tikzpicture}&
\begin{tikzpicture}[scale=2]
        \fill  [rounded corners=5pt, fill=black] (0,0)--(1,0)--(1,1)--(0,1)--cycle;
        
        \end{tikzpicture}&
\begin{tikzpicture}[scale=2]
        \fill  [rounded corners=5pt, fill=black] (0,0)--(1,0)--(1,1)--(0,1)--cycle;
        \fill  [rounded corners=15pt, fill=white] (0,0)--(1,0)--(1,1)--(0,1)--cycle;
        
        \end{tikzpicture}
\\
$\bun{A}$ & $\ext(D\bun{A})$ &$\bun{B}$ & $\ext(D\bun{B})$
\end{tabular}
\end{center}

\caption{Examples of the extended support for two indicator functions. }

\end{figure}


\section{Support stability outside the extended support}\label{sec:stab_extended_spt}


In this section, we  prove the main result of this paper, Theorem \ref{thm:spt_stability}, which shows that, under the source condition $\partial J(f)\neq \emptyset$, as $\la\to 0^+$ and $\normLdeux{w}/\la$ is small enough, almost all topological boundaries of the level sets of the solutions to ($\Pp_\la(f+w)$) converge towards the topological boundaries of the corresponding  level sets of $f$ in the sense of Hausdorff convergence. If,  moreover, $\normLdeux{w}/\la\to 0$, the support of $D\ulaw$ is contained in arbitrarily small tubular neighborhoods of the extended support $\ext(Df)$. 
In Section \ref{sec:stab_grad}, we show that the width of this tube can be further characterized through the knowledge of the vector field $z_0$ associated with $\voo$. We also observe  that an interesting consequence of our main result is that the minimal norm certificate $\voo$ is constant on each connected component of the extended support.

Throughout this section, we denote by $\vlaw$ the solution of $(\Dd_\la(f+w))$ and let $E_{\la,w}$ be any set of finite perimeter such that $\abs{\int_{E_{\la,w}} v_{\la,w} } = P(E_{\la,w})$. We also denote by $u_{\la,w}$ the solution of $\Pp_\la(f+w)$. Finally, let the level sets of $f$ be denoted by $F^{(t)}$ (refer to \eqref{eq:lev_sets} for the definition of level sets).

We begin by recalling an elementary result, which holds under very weak assumptions.

\begin{prop}\label{prop:l2cv}
  Let $f\in \LDD$ such that $J(f)<+\infty$. Let $\{w_n\}_{n\in\NN}$, $\{\la_n\}_{n\in\NN}$ be sequences such that $w_n\in \LDD$, $\normLdeux{w_n}\to 0$, and $\la_n\to 0^+$.

  Then $\lim_{n\to+\infty}\normLdeux{\un-f}=0$ and $\supp(Df)\subseteq \liminf_{n\to +\infty} \supp(D\un)$. 
\end{prop}

\begin{proof}
 For the sake of simplicity, we shall denote $\un$ by $\uno$. 

From the optimality of $\uno$, 
  \begin{equation*}
    \la_n J(\uno)+\frac{1}{2}\int_{\RR^2}(f+w_n-\uno)^2 \leq \la_n J(f),
  \end{equation*}
  we see that $\normLdeux{f-\uno}\to 0$ as $n\to+\infty$. Together with the fact that $J(\uno)\leq J(f)<+\infty$, that implies that $D\uno$ converges towards $Df$ in the weak-* topology of Radon measures. 
If $\supp(Df)=\emptyset$, there is nothing to prove. Otherwise, let $x\in \supp(Df)$. By weak-* convergence, for all $r>0$,
\begin{equation*}
0<\abs{Df}(B(x,r))\leq \liminf_{n\to+\infty} \abs{D\uno}(B(x,r)).
\end{equation*}
Hence,  $\limsup_{n\to +\infty}\dist(x,\supp(D\uno))\leq r$, and since this is true for all $r>0$, we obtain $x\in \liminf_{n\to+\infty} \supp(D\uno)$. This yields $\supp(Df)\subseteq \liminf_{n\to +\infty} \supp(D\uno)$.
\end{proof}

With the additional assumption that $\partial J(f)\neq \emptyset$, it is possible to describe the behavior of the level lines more precisely. In the following, we denote by $\Un{t}$ the $t$-level set of $u_{\la_n,w_n}$.

\begin{thm}\label{thm:spt_stability}
  Let $f\in \LDD$ such that $J(f)<+\infty$ and $\partial J(f)\neq \emptyset$.
  Let $\{w_n\}_{n\in\NN}$, $\{\la_n\}_{n\in\NN}$ be sequences such that $w_n\in \LDD$, $\la_n\to 0^+$, and ${\normLdeux{w_n}}/{\la_n}\leq \sqrt{\cD}/4$. .
  Then, up to a subsequence, for a.e.\ $t\in \RR$,
  \begin{align}\label{eq:limt}
 \lim_{n\to+\infty} \abs{\Un{t}\Delta \F{t}} =0,\qandq \lim_{n\to+\infty} \partial \Un{t}=\partial \F{t},
\end{align}
where the last limit holds in the sense of Hausdorff convergence.

If additionally, ${\normLdeux{w_n}}/{\la_n}\to 0$ as $n\to +\infty$, the full sequence satisfies
\begin{equation}\label{eq:suppliminfsup}
  \limsup_{n\to+\infty} \supp(D\un)\subseteq \ext(Df).
\end{equation}
\end{thm}

\begin{rem}
  It is possible to reformulate~\eqref{eq:suppliminfsup} in the following way. By Lemma~\ref{lem:largeball}, there exists $R>0$ such that for all $n$, $\supp(D\un)\subseteq B(0,R)$ and $\ext(Df))\subseteq B(0,R)$ so that by~\cite[Thm. 4.10]{rockafellarwets}), \eqref{eq:suppliminfsup} is equivalent to
 \begin{itemize}
   \item \textit{(outer limit inclusion)} for all $r>0$, there exists $n_0\in \NN$ such that,
     \begin{equation*}
  \forall n\geq n_0,\quad     \supp(D\un) \subseteq T_r \eqdef  \enscond{x\in \RR^2}{\mathrm{dist}(x,\ext(Df))\leq r }.
     \end{equation*}
   \item \textit{(inner limit inclusion)} for all $r>0$, there exists $n_1\in \NN$ such that,
     \begin{equation*}
  \forall n\geq n_1,\quad    \supp(Df) \subseteq  \enscond{x\in \RR^2}{\mathrm{dist}(x,\supp(D\un))\leq r }.
     \end{equation*}
 \end{itemize}
The second equation of~\eqref{eq:limt} has a similar reformulation.
\end{rem}


\begin{proof}

By~Lemma~\ref{lem:largeball}, there exists some radius $R>0$ such that for any $t\neq 0$, any $n\in \NN$, the level set $\Un{t}$ of $\uno$ is included in $B(0,R)$ (since $(\la_n,w_n)\in D_{1,\sqrt{\cD}/4}$). The same also holds for the level sets $\F{t}$ of $f$.
As a result, $\supp(\uno)\cup \supp(f)\subseteq \overline{B(0,R)}$ and the $L^2$ convergence of $\uno$ towards $f$ also implies its $L^1$ convergence. But by Fubini's theorem,
\begin{equation*}
  0=\lim_{n\to+\infty}  \int_{\RR^2}\abs{\uno-f}=\lim_{n\to+\infty} \int_{\RR} \abs{\Un{t}\triangle\F{t}}\d t, 
\end{equation*}
so that, up to the extraction of a subsequence $(\unp)_{n'\in\NN}$, for a.e. $t\in \RR$, $\lim_{n\to+\infty}\abs{\Unp{t}\triangle\F{t}}=0$.

Now let us fix such $t\in \RR$, and such a subsequence $(\unp)_{n'\in\NN}$. By $L^1$ convergence of $\bun{\Unp{t}}$ towards $\bun{\F{t}}$, and the fact that $\abs{D\bun{\Unp{t}}}(\RR^2)=P(\Unp{t})$ is uniformly bounded (by Lemma~\eqref{lem:unif_bd_lev_sets}), the gradient $D\bun{\Unp{t}}$ converges towards $D\bun{\F{t}}$ in the weak-* topology. Repeating the same argument as in Proposition~\ref{prop:l2cv} above, we obtain that 
\begin{equation*}
  \partial \F{t}=\supp(D\bun{\F{t}}) \subseteq \liminf_{n'\to+\infty}\supp(D\bun{\Unp{t}}) =\liminf_{n'\to+\infty} \partial\Unp{t}.
\end{equation*}


Let us prove that $\limsup_{n'\to+\infty} \partial\Unp{t}\subseteq \partial\F{t}$. If $\partial\Unp{t}=\emptyset$ for all $n'$ large enough, then $\limsup_{n'\to+\infty} \partial\Unp{t}=\emptyset$ and there is nothing to prove. Otherwise, let $(x_{n'})_{n\in\NN}$ such that $x_{n'}\in  \partial\Unp{t}$ and (up to the additional extraction of a subsequence - that we do not relabel) $\lim_{n'\to+\infty}x_{n'}=x\in \RR^2$. By Proposition~\ref{prop:weak_reg}, for all $r \leq r_0$,
\begin{equation*}
  \abs{B(x_{n'},r)\cap \Unp{t}}\geq \frac{1}{16}\abs{B(x_{n'},r)},\qandq \abs{B(x_{n'},r)\setminus \Unp{t}} \geq \frac{1}{16}\abs{B(x_{n'},r)}.
\end{equation*}
By the dominated convergence theorem, we obtain for $n\to+\infty$,
\begin{equation*}
  \abs{B(x,r)\cap \F{t}}\geq \frac{1}{16}\abs{B(x,r)},\qandq \abs{B(x,r)\setminus \F{t}} \geq \frac{1}{16}\abs{B(x,r)}.
\end{equation*}
Since this holds for all $r\in (0,r_0]$, we see that $x\in \partial\F{t}$, hence $\limsup_{n'\to+\infty} \partial\Unp{t}\subseteq \partial\F{t}$.

To prove $\limsup_{n\to+\infty} \supp(D\uno)\subseteq \ext(Df)$, we consider the full sequence again and we now assume that ${\normLdeux{w_n}}/{\la_n}\to 0$ as $n\to +\infty$. We denote by $\vno$ the dual certificate $v_{\la_n,w_n}$. If $\supp(D\uno)=\emptyset$ for all $n'$ large enough, there is nothing to prove. Otherwise, let $(x_n)_{n\in\NN}$ such that $x_n\in\supp(D\uno)$ and (up to the extraction of a subsequence) $\lim_{n\to+\infty}x_n=x$ for some $x\in \RR^2$. By Proposition~\ref{prop:prelim-suppdf}, it is not restrictive to assume that $x_n\in \partial E_n$ for some $E_n\in \FFn$ (otherwise we may replace $x_n$ with $y_n\in \partial E_n$ such that $\abs{x_n-y_n}\leq 1/n$).

  By Lemma~\ref{lem:unif_bd_lev_sets} and~\ref{lem:largeball}, the family $\{E_n\}_{n\in\NN}$ is relatively compact in the $L^1$ topology (see \cite[Thm. 12.26]{maggi2012sets}), that is, there exists $E\subseteq \RR^2$ with finite measure such that, up to the extraction of a subsequence, $\lim_{n\to +\infty} \abs{E\triangle E_n}=0$ (we do not relabel the subsequence). Moreover, up to the additional extraction of a subsequence, we may assume that either for all $n$, $\int_{E_n}\vno =P(E_n)$, or for all $n$,
$\int_{E_n}\vno =-P(E_n)$. We deal with the first case, the other being similar.

Passing to the limit in the optimality equation for $E_n$, we get
\begin{equation}\label{lsc_perim}
  P(E)\leq \liminf_{n\to +\infty}P(E_n)=\lim_{n\to+\infty}\int_{E_n}\vno =\int_{E}\voo,
\end{equation}
by the lower semi-continuity of the perimeter, and since $\bun{E_n}$ (resp. $v_{\la_n,w_n}$) converges strongly in $\LDD$ towards $\bun{E}$ (resp. $\voo$). Since $\voo\in \partial J(f)\subseteq \partial J(0)$, the converse inequality also holds, so that $P(E)=\int_E\voo$,  and $E\in \FFoo$. By definition of the extended support, this means that $\partial^*E\subseteq \ext(Df)$, hence $\partial E\subseteq \ext(Df)$.

Simarly as above, we conclude that $x\in\partial E$ in the following way. By Proposition~\ref{prop:weak_reg}, for all $r \leq r_0$,
\begin{equation}\label{eq:reg_conseq}
  \abs{B(x_n,r)\cap E_{n}}\geq \frac{1}{16}\abs{B(x_n,r)},\qandq \abs{B(x_n,r)\setminus E_{n}} \geq \frac{1}{16}\abs{B(x_n,r)}.
\end{equation}
By the dominated convergence theorem, we obtain for $n\to+\infty$,
\begin{equation*}
  \abs{B(x,r)\cap E}\geq \frac{1}{16}\abs{B(x,r)},\qandq \abs{B(x,r)\setminus E} \geq \frac{1}{16}\abs{B(x,r)}.
\end{equation*}
Since this holds for all $r\in (0,r_0]$, we see that $x\in \partial E\subseteq \ext(Df)$. Hence $\limsup_{n\to+\infty} \supp(D\uno)\subseteq \ext(Df)$.
\end{proof}

%
%
%
%
%
%
%
%
%

\begin{rem}[On dimensions $N\geq 3$]
The are two key elements to the proof of Theorem \ref{thm:spt_stability}: \begin{enumerate}
\item  \textbf{Compactness.} Lemma~\ref{lem:unif_bd_lev_sets} and~\ref{lem:largeball} which give that there exists $R,L>0$ such that $P(E_n)<L$ and the fact that there exists $R$ such that $E_n\subset B(0,R)$. This  allows the required compactness result to be applied.
\item \textbf{Weak Regularity.} Proposition \ref{prop:weak_reg} which ensures that the boundaries of all level sets are uniformly weakly  regular.
\end{enumerate}
The difficulty with extending Theorem \ref{thm:spt_stability} to higher dimensions is that the second property of weak regularity is no longer true: In dimension $N$, for weak regularity, we would require that $$\lim_{\la,\normLdeux{w}/\la\to 0}\norm{v_{\la,w}-\voo}_{L^N}=0.$$ However, the natural topology for $\{\vlaw\}_{\la,w}$ is $L^2(\RR^N)$ and when $N\geq 3$, there is no guarantee that the boundaries of the level sets of $\ulaw$ do not have arbitrarily many singular points such as cusps, and it may be the case that there are level sets of $\ulaw$ arbitrarily far out with arbitrarily small measure and perimeter.

When $N\geq 3$, it is still true that there exists $L$ such that $P(E)\leq L$ for all $E\in\FFlaw$ with $(\la,w)\in D_{1,\sqrt{c_2}/4}$ and it is  possible to adapt the argument in the proof of Theorem \ref{thm:spt_stability} to conclude that for each $r>0$,
$$\lim_{(\la_0,\alpha_0)\to (0,0)} \sup \enscond{\HDU{\partial^*\Elaws\setminus T_r}}{\Elaws \in \FFlaw,\ (\la,\alpha)\in \lnr{\alpha_0}{\la_0}}=0.$$
However, we have no guarantee that there exists $\la_0,\al_0>0$ such that $\Hh^{N-1}(\partial^* E\setminus T_r) = 0$ for all $E\in \Ff_{\la,w}$ with $(\la,w)\in D_{\la_0,\al_0}$.
\end{rem}

As consequence of the support stability theorem, we obtain the following result on the minimal norm certificate.
\begin{cor}
Let $\Omega$ be any connected component of $\RR^2\setminus \mathrm{Ext}(f)$. Then $\voo$ is constant on $\Omega$. 
\end{cor} 
\begin{proof}
For $\delta>0$, let $\Omega_\delta = \{x\in\Omega : \mathrm{dist}(\partial \Omega, x)>\delta\}$.
From Theorem  \ref{thm:spt_stability}, we know that for all $\delta$, there exists $\la_\delta>0$ such that for all $\la\in D_{\la_\delta, 0}$, $u_{\la,0}$ is constant on $\Omega_\delta$. Since $v_{\la,0} = (f-u_{\la,0})/\la$, it follows that $v_{\la,0}$ is also constant on $\Omega_\delta$. So, since $\voo$ is the $L^2$ limit of $v_{\la,0}$, $\voo$ must be constant on $\Omega_\delta$ for all $\delta>0$. Therefore, $\voo$ is constant on $\Omega$.
\end{proof}


\section{Support stablity for nonsmooth convex sets}\label{sec:no_source_cond}

The theory developed in Sections~\ref{sec:proplev}, \ref{sec:extended} and \ref{sec:stab_extended_spt} relies on the existence of a subgradient of the total variation for $f$ in the $L^2$ topology (source condition). As natural as it may seem, this hypothesis does not always hold even for simple signals (like the indicator function of a square). In some cases, however, there is a natural limit for the dual certificates $\vlao$ when considering another topology.

 This section studies the case of a union $\Omega$ of disjoint convex subsets of $\RR^2$ which are sufficiently far apart. If their boundary is not smooth enough, the source condition is not satisfied. Still, we shall prove that one can  guarantee support stability for the solutions of \eqref{eq-rof}. A notable example is the unit square where $\Omega=[0,1]^2$.

 As usual, througout this section, we let $\ulaw$ be the solution of ($\Pp_\la(f+w)$) and $\vlaw$ be the solution of ($\Dd_\la(f+w)$). We also recall the notation from Section~\ref{sec:duality} where
given any bounded open convex set $C$,   $C_\rho$ is the opening of $C$ by open balls of radius $\rho$ and there exists a unique function $r(x)$ such that $x\in\partial C_{r(x)}$ and $x$ belongs to an arc of a circle of radius $r(x)$.

\subsection{Dual certificates for unions of convex sets}\label{sec:dual_cert_cvx_sets}
Let $C$ be a bounded open convex subset of $\RR^2$. The dual certificate $\vlao$ associated with $f=\bun{C}$ is given in~\eqref{eq:convexcertif}.

Now, more generally, if $f = \bun{\Omega}$, where  $\Omega = \cup_{j=1}^M C^{(j)}$ and $\{C^{(j)}\}_{j=1}^M$  are bounded open convex sets such that given any $0\leq k\leq M$ and any permutation $\{i_1,\ldots, i_M\}$ of $\{1,\ldots, M\}$,
$$
E_{i_1,\ldots, i_k}\in \argmin\enscond{P(E)}{ P(E)<\infty, ~ \bigcup_{j=1}^k C^{(i_j)}\subset E\subset \RR^2\setminus \bigcup_{j=k+1}^M C^{(i_j)}}
$$
implies that  $P(E_{i_1,\ldots, i_k}) > \sum_{j=1}^k P(C^{(i_j)})$, then, as proved in \cite{altercalib05,alterconvex05}, the solution $u_{\la,0}$ to \eqref{eq-rof} is
$$
u_{\la,0} = \sum_{j=1}^M \left(1+\la v_{C^{(j)}} \right)^+ \bun{C^{(j)}},
$$
and consequently,
$$
v_{\la,0}(x) = \begin{cases}
v_{C^{(j)}}(x) & x\in C_\la^{(j)}, ~j=1,\ldots, M \\
1/\la & x\in C^{(j)}\setminus C_\la^{(j)}, ~j=1,\ldots, M\\
0& \text{otherwise.}
\end{cases}
$$

While $\lim_{\la\to 0}\norm{v_{\la,0}}_{L^2} = +\infty$, we observe that the function $$\voo \eqdef \sum_{j=1}^M v_{C^{(j)}}\in L^1(\RR^2).$$ Indeed, for each $j$, by the monotone convergence theorem
\begin{equation*}
\normLun{ v_{C^{(j)}}} = \int_{\RR^2} v_{C^{(j)}} =\lim_{n\to+\infty} \int_{\RR^2} v_{C^{(j)}}\bun{C^{(j)}_{1/n} } = \lim_{n\to+\infty}P(C^{(j)}_{1/n})=P(C^{(j)})<+\infty,
\end{equation*}
where $C^{(j)}_{1/n}$ denotes the opening of $C^{(j)}$ with radius $1/n$.

Moreover, since $v_{\la,0}\geq v_{\mu,0}$ for $\mu\geq\la$ and since for a.e. $x\in \RR^2$,
$$
\lim_{\la\to 0}v_{\la,0}(x) = \voo(x),
$$
if follows by the monotone convergence theorem that
\begin{equation*}
\norm{v_{\la,0} - \voo}_{L^1}\to 0, \qquad \la\to 0.
\end{equation*}
 As before, we may define the extended support of $f$ via $\voo$ as
\begin{align}\label{eq:ext_spt_cvx_sets}
\begin{split}
  \ext(Df)&\eqdef \overline{\enscond{\partial^* E}{\pm\int_E \voo = P(E),~ \abs{E}<\infty}}\\
& = \Omega\setminus \left(\bigcup_{j=1}^M \mathrm{int}\left(C^{(j)}_{R_j}\right)\right),
\end{split}
\end{align}
where $C^{(j)}_{R_j}$ is the maximal calibrable set inside $C^{(j)}$ for each $j=1,\ldots, M$. We remark that $v_{\la,0} = \voo$ on $\RR^2 \setminus \ext(Df)$.

\begin{rem}
In the limit case where equality may hold in $P(E_{i_1,\ldots, i_k}) \geq \sum_{j=1}^k P(C^{(i_j)})$, the extended support of $\bun{\Omega}$ may be larger than $\partial \Omega$. The case $\Omega = B(x_1,R)\cup B(x_2,R)$ and $\abs{x_1 - x_2} = \pi R$, is shown in Figure \ref{fig:twoBalls}. In this case, if $E$ is the convex hull of $\Omega$, then $P(E) = P(\Omega)$. In the absence of noise, the support of any TV regularized solution is simply $\partial \Omega$, however, the extended support is strictly larger than $\partial \Omega$. This is essentially reflected in the fact that the presence of any noise which shifts the two balls towards each other will necessarily result in additional level lines. We refer to~\cite{Allard3,caselles2012tv} for a detailed study of this example.
\end{rem}

\begin{figure}[H]
\begin{center}
\begin{tabular}{c@{\hspace{50pt}}c}
\begin{tikzpicture}[scale=2]

\fill[black] (.5,.5) circle(5mm);
\fill[black] (pi*.5+.5,.5) circle(5mm);
\end{tikzpicture}&
 \begin{tikzpicture}[scale=2]

\draw[black] (.5,.5) circle(5mm);
\draw[black] (pi*.5+.5,.5) circle(5mm);
\draw (.5,1) -- (pi*.5+.5,1);
\draw (.5,0) -- (pi*.5+.5,0);
\end{tikzpicture}\\
$\bun{\Omega}$ &$\ext(D \bun{\Omega})$
\end{tabular}
\end{center}
\caption{The extended support for the indicator function of a union of two balls.}
\label{fig:twoBalls}
\end{figure}
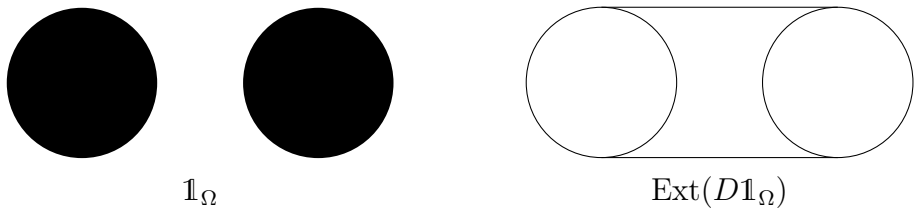

\subsection{Support stability}
In this section, we prove that the support of $\ulaw$ is stable around the extended support \eqref{eq:ext_spt_cvx_sets}, i.e.  its support is contained inside some neighborhood of $\ext(Df)$, 
whenever $\la$ and $ \normLdeux{w}/\la$ are sufficiently small.
We begin by proving some properties of the level sets.
\begin{prop}\label{prop:L1_props}
The following statements are true.
\begin{itemize}
\item[(i)] There exists $\al_0,\la_0,L >0$ such that $P(E)\leq L$ for all $E\in \FFlaw$ and $(\la,w)\in D_{1,\sqrt{\cD}/4}$.
\item[(ii)] There exists $R>0$ such that $E_{\la,w}\subset B(0,R)$ for all $\Elaw\in \FFlaw$ with $(\la,w)\in D_{1,\sqrt{\cD}/4}$.
\end{itemize}
\end{prop}
\begin{proof}

To prove (i), 
recall from the discussion in Section \ref{sec:dual_cert_cvx_sets} that $\norm{v_{\la,0}}_{L^2} \leq \norm{\voo}_{L^1}$.
So, for all $E\in\FFlaw$ with $(\la,w)\in D_{1,\sqrt{c_2}/4}$,
 \begin{align*}
 P(E) = \pm\int_E \vlaw \leq \frac{\normLdeux{w}\abs{E}^{1/2}}{\la}+ \norm{v_{\la,0}}_{L^1}
 \leq \frac{ P(E)}{4}  + \norm{\voo}_{L^1},
 \end{align*}

 For the proof of (ii) is very similar to the proof of Lemma \ref{lem:largeball}. We first show that there exists $R>0$ such that $E_{\la,w}\cap B(0,R)\neq \emptyset$ for all $\Elaw\in\FFlaw$ with $(\la,w)\in D_{1,\sqrt{c_2}/4}$: let $R$ be such that $B(0,R)\supset \Omega$. For a contradiction, suppose that $\emptyset\neq E_{\la,w}\subset B(0,R)^c$. Then, since $v_{\la,0} = 0$ on $\Omega^c$, we have that
$$
P(E_{\la,w}) = \int_{E_\la,w}(v_{\la,w} - v_{\la,0})  \leq \frac{\norm{w}_{L^2} \abs{E_{\la,w}}^{1/2}}{\la}< \frac{P(E_{\la,w})}{4},
$$ 
which is impossible if $P(E_{\la,w}) > 0$. So, $E_{\la,w}\cap B(0,R)\neq \emptyset$ if $\Elaw\neq\emptyset$. Finally, since by (i), there exists $L>0$ such that $P(\Elaw)\leq L$ for all $(\la,w)\in D_{1,\sqrt{\cD}/4}$,  it follows that $\Elaw\subset B(0,R+L)$.

\end{proof}

When the source condition is not satisfied,  Proposition \ref{prop:weak_reg} cannot be applied directly to the level sets of $\ulaw$.  However, even if the source condition does not hold, there may still be a subset of $V$ of $\RR^2$ for which
$$
\lim_{\la\to 0}\norm{v_{\la,0}-\voo}_{L^2(V)}= 0.$$ In this case, one can argue along the lines of Proposition \ref{prop:weak_reg} to deduce that  there is still weak regularity on a subset of $\partial E$.  Note that for  characteristic functions on unions of convex sets as described in Section \ref{sec:dual_cert_cvx_sets}, we can let $V=\RR^2 \setminus \ext(Df)$ since $v_{\la,0} = \voo$ on $\RR^2\setminus \ext(Df)$. The precise regularity statement is given in the following proposition.
\begin{prop}\label{prop:weak_reg_partial}
Let $V\subset \RR^2$ be an open set. Suppose that
$$
\lim_{\la\to 0}\norm{v_{\la,0}-\voo}_{L^2(V)} = 0.
$$
Then, there exists $r_0>0$ such that for all $r\in [0, r_0]$ and $\Elaw\in\FFlaw$ with $(\la,w)\in D_{1,\sqrt{\cD}/4}$, if $x\in \partial \Elaw$ is such that $B(x,r_0)\subseteq V$, then
\begin{equation*}
 \frac{\abs{B(x,r)\cap \Elaw}}{\abs{B(x,r)}}\geq \frac{1}{16}\qandq \frac{\abs{B(x,r)\setminus \Elaw}}{\abs{B(x,r)}} \geq \frac{1}{16}.
\end{equation*}
\end{prop}

\begin{thm}\label{thm:union_cvx_sets}
  Let $\Omega$ be a union of convex sets which satisfies the assumptions of Section~\ref{sec:dual_cert_cvx_sets}. 
    Let $\{w_n\}_{n\in\NN}$, $\{\la_n\}_{n\in\NN}$ be sequences such that $w_n\in \LDD$, $\la_n\to 0^+$, and ${\normLdeux{w_n}}/{\la_n}\leq \sqrt{\cD}/4$. .
  Then, up to a subsequence, for a.e.\ $t\in \RR$,
  \begin{align}\label{eq:limtbis}
    &\lim_{n\to+\infty} \abs{\Un{t}\Delta \F{t}} =0,\qandq \\
    &\begin{cases}
   \lim_{n\to+\infty} \partial \Un{t}=\partial \F{t}, & \mbox{ if $0<t\leq 1$}\\
   \limsup_{n\to+\infty}  \partial \Un{t}\subseteq \partial \Omega & \mbox{otherwise,}\label{eq:limtter}
\end{cases}
\end{align}
where the last limits holds in the sense of Hausdorff convergence.

If additionally, ${\normLdeux{w_n}}/{\la_n}\to 0$ as $n\to +\infty$, the full sequence satisfies
\begin{equation}\label{eq:suppliminfsupsquare}
  \supp(Df)\subseteq \liminf_{n\to +\infty} \supp(D\un) \subseteq \limsup_{n\to+\infty} \supp(D\un)\subseteq \ext(Df).
\end{equation}

\end{thm}
\begin{proof}
  
  By Proposition~\ref{prop:L1_props} (ii), the level sets $\Un{t}$ are included in some ball $B(0,R)$. So, by the same argument as in Theorem~\ref{thm:spt_stability}, $\partial \F{t}\subseteq\liminf_{n'\to+\infty} \partial\Unp{t}$ and~\eqref{eq:limtbis} holds.

  Now, we prove $\limsup_{n'\to+\infty} \partial\Unp{t}\subseteq \partial \Omega$. Let $(x_{n'})_{n'\in\NN}$ be a sequence in $\RR^2$ such that (up to an additional extraction) $x_{n'}\to x\in \RR^2$, and we assume by contradiction that $x\notin  \partial \Omega$. We let 
  \begin{equation*}
    V\eqdef \enscond{y\in \RR^2}{\dist(y,\partial \Omega)>r_1},
  \end{equation*}
  where $r_1>0$ is such that $r_1<\dist(x,\partial \Omega)$. We observe that $V$ is open, $x\in V$ and $\lim_{n'\to+\infty} \norm{v_{n'}-\voo}_{L^2(V)} = 0$.

  Applying Proposition~\ref{prop:weak_reg_partial}, we obtain for $n'$ large enough and $r>0$ small enough, 
\begin{equation*}
  \abs{B(x_{n'},r)\cap \Unp{t}}\geq \frac{1}{16}\abs{B(x_{n'},r)},\qandq \abs{B(x_{n'},r)\setminus \Unp{t}} \geq \frac{1}{16}\abs{B(x_{n'},r)},
\end{equation*}
and we conclude, passing to the limit as in the proof of Theorem~\ref{thm:spt_stability} that $x\in \partial\F{t}\subseteq \partial \Omega$, which contradicts the hypothesis. Hence $x\in\partial\Omega$, and $\limsup_{n'\to+\infty} \partial\Unp{t}\subseteq \partial \Omega$. Equation~\eqref{eq:limtter} follows since $\partial\F{t}=\partial \Omega$ for $0<t\leq 1$.

It remains to prove $ \limsup_{n\to+\infty} \supp(D\un)\subseteq \ext(Df)$. The proof is quite similar to the proof presented for Theorem \ref{thm:spt_stability} and we merely sketch it for brevity.
We let $x_{n}\to x\in \RR^2$, where $x_n\in \partial E_n$ for some $E_n\in \FFn$ and we assume by contradiction that $x\notin \ext(Df)$. Arguing as  above, and using the compactness property provided by Proposition~\ref{prop:L1_props}, we see that $x\in \partial E$ where $E$ is the limit of $E_n$ (up to an extraction, for the $L^1$ topology).

To conclude, we need to prove that $E\in \FFoo$, as in~\eqref{lsc_perim}. The $L^2$ convergence of 
$v_{\la_n,w_n}$ towards $v_{0,0}$ was applied to prove equation (\ref{lsc_perim}), but in fact, $L^1$ convergence $v_{\la_n,0}$ is sufficient:
Note that by (i) of Proposition \ref{prop:L1_props}  and the isoperimetric inequality, there exists $C$ such that $\abs{E_n}\leq C^2$ for all $n$. Then by letting $E_n\triangle E = (E_n\setminus E) \cup (E\setminus E_n)$,
\begin{align*}
\abs{\int_{E_n} v_{\la_n,w_n} -\int_E \voo}&\leq \abs{ \int_{E_n} v_{\la_n,w_n} - \int_{E_n} \voo} +\abs{ \int_{E_n} \voo - \int_E \voo }\\
&\leq \norm{v_{\la_n,0}  - \voo}_{L^1} +\frac{C \normLdeux{w_n}}{\la_n} + \int_{E_n\triangle E}\voo \to 0,
\end{align*}
by the $L^1$ convergence of $(v_{\la,0})$ and the absolute continuity of the Lebesgue integral.
As a result,  $E\in \FFoo$, hence $x\in \partial E\subseteq \ext(Df)$, and $\limsup_{n\to+\infty} \supp(D\un)\subseteq \ext(Df)$.
\end{proof}


\section{Support stability and calibrations}\label{sec:stab_grad}
Theorem \ref{thm:spt_stability} shows that as a result of the strong $L^2$ convergence of $\vlaw$ to $\voo$, one is guaranteed support stability outside a small neighbourhood of the extended support. This section upper bounds the rate of convergence in the outer limit inclusion of \eqref{eq:suppliminfsup}. In particular, we make explicit the relationship between the width of this neighbourhood, the decay of $\normLdeux{\vlaw-\voo}$ and the nondegeneracy of $z_0$, the vector field for which $\voo = -\divx z_0$. .

\subsection{Support stability}
In this section, we define 
\begin{equation*}
T_r \eqdef  \enscond{x\in \RR^2}{\mathrm{dist}(x,\ext(Df))\leq r },
\end{equation*}
we make an additional assumption about the decay of $z_0$ away from the extended support: 
\begin{equation}
  \deltar\eqdef 1-\ess \sup_{x\in T_r^C} \abs{z_0(x)}>0.\label{eq:defer}
\end{equation}
We also let $C$ be such that
$$
\abs{E}\leq C^2, \qquad \forall~E\in\FFlaw, ~ (\la,w)\in D_{1,\sqrt{\cD}/4}.
$$
Recall that the existence of $ C$ is guaranteed by Lemma \ref{lem:unif_bd_lev_sets}.
Examples of  vector fields whose decay is known outside the extended support are described in Section~\ref{sec:calibrable}.

\begin{prop}\label{prop:partoutsidetube}
 Given any $E\in \FFlaw$ with $(\la,w)\in D_{1,\sqrt{\cD}/4}$,
\begin{equation*}
  \deltar \HDU{\partial^*E\setminus T_r}\leq \int_{E}(v_{\la,w} - \voo)\leq C\normLdeux{v_{\la,w}-\voo},
\end{equation*}
\end{prop}

\begin{proof}
Comparing the energy of $E$ with that of the empty set we get,
\begin{align*}
  P(E)\leq \int_{E}v_{\la,w} &= \int_{E}(v_{\la,w}-\voo) + \int_{E}\voo
  \\
                                 &\leq \normLdeux{v_{\la,w}-\voo}\sqrt{\abs{E}} + \int_{\partial^* E} z_0\cdot \nu                               .
\end{align*}
Recall from  Lemma \ref{lem:unif_bd_lev_sets} that  there exists $C>0$ such that $\abs{E}\leq C^2$. So, $$ P(E)- \int_{\partial^* E\cap T_r}z_0\cdot \nu - \int_{\partial^* E\setminus T_r}z_0\cdot \nu \leq C \normLdeux{v_{\la,w}-\voo}.$$
Since $\left(\ess \sup_{\RR^N\setminus T_r}\abs{z_0}\right)\leq 1-\deltar$, and more generally $\normLinf{z_0}\leq 1$, 
\begin{align*}
  P(E)- \int_{\partial^* E\cap T_r}z_0\cdot \nu - \int_{\partial^* E\setminus T_r}z_0\cdot \nu  &\geq \HDU{\partial^*E\setminus T_r}+\HDU{\partial^*E\cap T_r}\\
                                                                                                            &- (1-\deltar)\HDU{\partial^*E\setminus T_r} - \HDU{\partial^*E\cap T_r}\\
                                                                                                            &\geq  \deltar \left(\HDU{\partial^*E\setminus T_r}\right),
\end{align*}
hence the claimed result.

\end{proof}

\begin{prop}\label{prop:entiretubecomplement}
Let $\la>0$ and $w\in \LDD$ be such that $\normLdeux{v_{\la,w}-v_0} <\deltard\sqrt{\cD}$. Then, $E\in \FFlaw$ and $P(E)>0$ implies that
  \begin{equation*}
    \HDU{\partial^* E \cap T_{r/2}}>0.
  \end{equation*}
\end{prop}
\begin{proof}
Assume by contradiction that $P(E)>0$ and $\HDU{\partial^* E\cap T_{r/2}}=0$ so that $\partial^* E \subset T_{r/2}^c$ up to an $\Hh^1$-negligible set. Then,
\begin{align*}
P(E) &= \int_E \vlaw \leq \norm{E}^{1/2} \normLdeux{\vlaw - \voo} + \int_E \voo\\
&\leq \frac{P(E) \normLdeux{\vlaw - \voo} }{\sqrt{\cD}}+ (1-\delta_{r/2}) P(E)
\end{align*}
where we have applied the isoperimetric inequality and the fact that $\voo = \divx z_0$ with $\abs{z_0}\leq (1-\delta_{r/2})$ on $T_{r/2}^c$. Since $P(E)>0$, this implies that
$$
\delta_{r/2}\sqrt{\cD} \leq \normLdeux{\vlaw - \voo}.$$
This contradicts the assumption of this proposition.
\end{proof}

\begin{thm}\label{thm:stab_w_vec_field}
Let $r>0$.
 If $(\la,w)\in D_{1,\sqrt{\cD}/4}$ are such that
  \begin{equation}\label{eq:thm_vec_conds}
  \norm{v_{\la,w} - v_0}_{L^2}\leq \delta_{r/2} \min \left\{\frac{r}{2C}, \sqrt{\cD}\right\},
  \end{equation}
  then for all level sets $E$ of $u_{\la,w}$, 
\begin{equation*}
  \partial E \subseteq T_r.
\end{equation*}

\end{thm}

\begin{proof}
  It is sufficient to show that $\Hh^1(\partial^* E \setminus T_r) = 0$ for all $E\in\FFlaw$ with $\la>0$ and $w\in \LDD$ satisfying (\ref{eq:thm_vec_conds}). For, if we have $\Hh^1(\partial^* E \setminus T_r) = 0$, this means that $\bun{E}$ is constant on every connected component of the open set $\RR^2\setminus T_r$, hence the topological boundary satisfies $\partial E\subseteq T_r$.
Furthermore, by Section~\ref{sec:jordan_decomp}, we may assume that  up to an $\Hh^1$-negligible set, $\partial^* E$ is equivalent to a Jordan curve $J$.

First observe that by Proposition \ref{prop:entiretubecomplement}, $\Hh^1(\partial^* E\cap T_{r/2})>0$. Now, for a contradiction, suppose that $\Hh^1(\partial^* E\setminus T_{r})>0$. Then since this implies that $\Hh^1(J\cap T_{r/2})>0$, $\Hh^1(J\setminus T_r)>0$ and $J$ is a continuous curve, it follows that $$\Hh^1(J \setminus T_{r/2}) = \Hh^1(\partial^* E \setminus T_{r/2})> r/2.$$ However, this is a contradiction Proposition \ref{prop:partoutsidetube} implies that
$$\lim_{(\la_0,\alpha_0)\to (0,0)} \sup \enscond{\HDU{\partial^*\Elaws\setminus T_r}}{\Elaws \in \FFlaw,\ (\la,\alpha)\in \lnr{\alpha_0}{\la_0}}=0.$$
Indeed, by our choice of $(\la,w)$ in (\ref{eq:thm_vec_conds}), if $\Hh^1(\partial^* E\setminus T_r)>0$, then the combination of Proposition  \ref{prop:partoutsidetube} and Proposition \ref{prop:entiretubecomplement} yields
\begin{equation*}
\frac{r\delta_{r/2}}{2} < \HDU{\partial^*E\setminus T_{r/2}}\leq C\normLdeux{v_{\la,w}-\voo} \leq \frac{r \delta_{r/2}}{2}.
\end{equation*}

\end{proof}

\paragraph{Example}
In the case where $f = \bun{B(0,R)}$, by  the construction of $z_0$ from  (\ref{eq-calib-disc}), $\delta_r\leq r/R$.
Furthermore, since $\vlaw = \frac{2}{R}f$, for each $E\in \FFlaw$,
$$
2\sqrt{\pi}\abs{E}^{1/2}\leq P(E) \leq \int_E \vlaw
\leq 2\pi R + \norm{w}_{L^2}\abs{E}^{1/2},
$$
and $\abs{E}^{1/2}\leq 2\pi R/(2\sqrt{\pi} - \normLdeux{w})$ provided that $2\sqrt{\pi} > \normLdeux{w}$.
 So, Theorem \ref{thm:stab_w_vec_field} implies that for all $\la>0$, and $w\in\LDD$ such that $$\norm{w}_{L^2}\leq \min\left\{2\sqrt{\pi}-\pi,~\frac{\la  r^2}{8R^2}\right\},$$ any level set $E$ of $\ulaw$ satisfies $\partial^* E\subset T_{r}$ up to an $\Hh^1$-negligible set.

\subsection{Non-degeneracy of calibrable sets}\label{sec:calibrable}

The aim of this section is to show that if $C\subset\RR^2$ is a convex calibrable set, the minimal norm certificate $\voo=h_C\bun{C}$ (where $h_C=\frac{P(C)}{\abs{C}}$) can be written as $\voo=\divx z_0$ where $z_0\in \XDD$, $(z,D\bun{C})=-\abs{D\bun{C}}$ and for every compact set $K\subset \RR^2\setminus \partial C$,
\begin{equation*}
  \ess \sup_{K} \abs{z_0} <1,
\end{equation*}
with an estimation on that inequality. We do not aim at full generality, and we assume that $\partial C$ is of class $\Cder{2}$ for the sake of simplicity. Reducing the hypotheses is the subject of future work.

The proof relies on the notion of inner and outer calibrations described in~\cite{beltvflow02}, which amounts to constructing vector fields ``inside'' and ``outside'' the studied set, and then ``glue'' the two constructions.

\begin{defn}
Let $C \subseteq \RR^2$ be a set of finite perimeter. We say that $C$ is $-$calibrable if there exists a vector field $z^-_C: \RR^2\rightarrow \RR^2$ 
such that 
\begin{enumerate}
  \item $z^-_C\in L^2_{loc}(\RR^2,\RR^2)$ and $\divx z^-_C \in L^2_{loc}(\RR^2)$;
  \item $|z^-_{C}| \leq 1$ almost everywhere in $C$;
  \item $\divx z^-_C$ is constant on $C$;
  \item $\theta(z^-,-D\bun{C})(x)=-1$ for $\Hh^1$-almost every $x\in \partial^* C$.
\end{enumerate}
Similarly $C$ is $+$calibrable if $1),2),3)$ hold and $\theta(z^+,-D\bun{C})(x)=+1$ in $4)$.
\end{defn}
The following lemma tells that one may ``glue'' calibrations:
\begin{lem}[\protect{\cite{beltvflow02}}]
  Let  $C$ be a bounded set of finite perimeter. Then $v=\bun{C}$ is calibrable if and only if $C$ is $-$calibrable with $-\divx \xi^-_{C}=h_C$ in $C$, and $\RR^2\setminus C$ is $+$calibrable with $\divx z^+_{\RR^2}=0$ in $\RR^2\setminus C$,
  defining
  \begin{align*}
    z\eqdef \left\{ \begin{array}{cc}
        z^-_C & \mbox{on }  C,\\
        z^+_C & \mbox{on } \RR^2\setminus C.
    \end{array}\right.
  \end{align*}
\end{lem}

\subsection{Inner calibrations}
Let $C\subset \RR^2$ be a bounded open convex set of class $\Cder{2}$ 
, and $h_C \eqdef \frac{P(C)}{C}$.

Following~\cite{AmaBel15} in order to build the calibration, we consider the following auxiliary problem:
\begin{equation}
  \divx\left(\frac{\nabla u}{\sqrt{1+\abs{\nabla u}^2}}\right) = h_C. \label{eq:pscbmean}
\end{equation}
and we define
\begin{equation}\label{eq:zinnercalib}
  z\eqdef\begin{cases}
\frac{\nabla u}{\sqrt{1+\abs{\nabla u}^2}} &\mbox{on $C$}\\
\nu_C & \mbox{on $\partial C$}\end{cases}
\end{equation}

Giusti proved the following result in~\cite{giusti78} (see also~\cite[Prop.6.2]{AmaBel15})
\begin{thm}[\cite{giusti78}]\label{thm:giustiexist} There exists a solution $u\in\Cder{2}(C)$ to~\eqref{eq:pscbmean} if and only if 
  \begin{equation}
    \forall B\subsetneq C, B\neq \emptyset, \quad h_C < \frac{P(B)}{\abs{B}}. \label{eq:uniqcheeg}
  \end{equation}
That solution $u$ is unique up to an additive constant, bounded from below in $C$, and its graph is vertical at the boundary of $C$, in the sense that 
  \begin{equation*}
    \frac{\nabla u}{\sqrt{1+\abs{\nabla u}^2}} \to \nu_C \mbox{ uniformly on }\partial C.
  \end{equation*}
\end{thm}
The consequence is that $z$ defined in~\eqref{eq:zinnercalib} is a $\Cder{1}$ vector field in $C$, (in fact analytic, see~\cite{AmaBel15}), continuous in $\overline{C}$.

In fact, Giusti also proved that the condition~\eqref{eq:uniqcheeg} is equivalent to~\eqref{eq:prelimcalibchar}, namely the calibrability of $C$ (this result was extended to $\RR^N$ in \cite{alteruniq09}). As a result, for a calibrable set $C$, one may choose the calibration given by the vectorfield $z$ such that
\begin{equation*}
  \forall x\in C,\quad  z(x)=\frac{\nabla u(x)}{\sqrt{1+\abs{\nabla u(x)}^2}}
\end{equation*}
and $\restr{z}{\RR^2\setminus C}$ is a vectorfield such that $\normi{z}\leq 1$, $\divx z=h_C$ and $\theta(z,D\bun{C})=-1$.

A first step in proving that $\abs{z}<1$ inside $C$ is the following theorem by Giusti.
\begin{thm}[\cite{giusti78}]\label{thm:giustibounded}
  For every compact set $K\subset \interop{C}$, there exist exists $Q>0$ such that for any solution of~\eqref{eq:pscbmean} in $\interop{C}$, 
  \begin{equation*}
    \sup_K\abs{\nabla u}\leq Q.
  \end{equation*}
\end{thm}
This implies that $\sup_{K}\frac{\abs{\nabla u}}{\sqrt{1+\abs{\nabla u}^2}}<1$. In the next proposition, we study further its decay  inside $C$, which yields a non-degenerate inner calibration for $C$.

\begin{prop}
  Let $C\subset \RR^2$ be a bounded strictly convex calibrable set such that $\partial C$ is of class $\Cder{2}$ 
  and $h_C >\sup_{\partial C} \abs{\kappa_{\partial C}}$. 
  Assume moreover that the solution to~\eqref{eq:pscbmean} is continuous up to the boundary, i.e.\ $u\in \Cder{}(\overline{C})$.
Then, there exists $\alpha>0$, there exists a vector field $z\in \Cder{}(\overline{\Omega})\cap \Cder{1}(\Omega)$ such that $\divx z= h_C$, $z\cdot \nu =1$ on $\partial C$, and
  \begin{equation*}
    \forall x\in C, \quad \abs{z(x)}\leq \frac{\alpha}{\sqrt{d(x)^2+ \alpha^2}}, 
  \end{equation*}
  where $d(x)\eqdef \mathrm{dist}(x,\partial C)$.
\end{prop}
\begin{proof}

  By~Theorem~\ref{thm:giustiexist}, there exists a $\Cder{2}$ solution $u$ to~\eqref{eq:pscbmean} which is vertical at the boundary, and the inequality $h_C > \ess \sup_{\partial C} \abs{\kappa_{\partial C}}$ implies that $u$ is bounded (see~\cite[Th. 3.1]{giusti78}). We define $z(x)\eqdef \frac{\nabla u(x)}{\sqrt{1+\abs{\nabla u(x)}^2}}$ for all $x\in C$. 

  Let us prove that $\abs{\nabla u(x)}>0$ for a.e.\ $x\in C$.
  First, we assume that $C$ is strictly convex. Since $u\in  \Cder{2}{(\Omega)}\cap\Cder{}{(\overline{\Omega})}$, by \cite[Th.~2.2]{korevaar1983convex} u is a convex function. As a result, $\enscond{x\in C}{\nabla u(x)=0}=\argmin_C u$, and it is thus a closed convex set. Assume by contradiction that the dimension of $\argmin_C u$ is $2$, i.e.\ $\argmin_C u$ contains an open ball $B(x_0,r) \subset \interop(C)$ for some $x_0\in \interop(C)$, $r>0$. Let $T$ denote the operator $T:u\mapsto \frac{\nabla u}{\sqrt{1+\abs{\nabla u}^2}}$, and let $w$ be the constant function $x\mapsto \min_C u$.  We have $u\leq w$ in $\partial B(x_0,r)$ (in fact equality holds), and $0=\divx Tw< \divx Tu=h_C$ in $B(x_0,r)$. By Theorem~\ref{thm:giustibounded}, Problem~\eqref{eq:pscbmean} is locally uniformly elliptic, and the comparison principle~\cite[Th.~10.1]{gilbarg1977elliptic} yields that $u<w$ in $B(x_0,r)$, which is a contradiction. As a result, the dimension of $\argmin_C u$ is strictly less than $2$ and $\argmin_C u$ is Lebesgue-negligible.

  Now, for a.e.\ $x\in \interop(C)$, we may define $y\eqdef x+d(x)\frac{\nabla u(x)}{\abs{\nabla u(x)}}$, and we observe that $y\in C$. By convexity of $u$, $u(y)-u(x)\geq \nabla u(x)\cdot(y-x)=d(x)\abs{\nabla u(x)}$. As a result, $\abs{\nabla u(x)}\leq \frac{2\normi{u}}{d(x)}$, and
  \begin{equation*}
    \abs{Tu(x)}\leq \frac{2\normi{u}}{\sqrt{d(x)^2+4\normi{u}^2}}. 
  \end{equation*}
The claimed result holds by a density argument.  
\end{proof}

\subsection{Outer calibrations}
\label{sec:calibout}
It is proved in~\cite[Th. 5]{beltvflow02} (see also \cite[Th. 13]{altercalib05} in dimension $N$) that sets which satisfy a geometric condition (namely convex sets that are far enough from one another) have a complement which is $+$calibrable. That condition holds for $\Cder{1,1}$ convex sets.

However, it is not clear from the proof that the corresponding vector field has norm $<1$ in compact sets of $\RR^2\setminus C$.
We provide below an explicit construction when the set has $\Cder{2}$ boundary. Admittedly the hypothesis is quite restrictive but we think that this construction gives some insight on the geometric properties involved.

\begin{prop}
Let $C\subset \RR^2$ be a nonempty bounded open convex subset with $\Cder{2}$ boundary.
There exists a vector field $z\in L^\infty\cap \Cder{}{(\overline{\RR^2\setminus C})}$ such that $z=\nu$ on $\partial C$, $\divx z=0$ in the sense of distributions and $\abs{z}<1$ on every compact subset of $\RR^2\setminus C$.
\end{prop}
The decay of $z$ is discussed in Remark~\ref{rem:decayz} below.

\begin{proof}
 We choose an arclength parametrization of $\partial C$, $s\mapsto y(s)$ defined on $S\eqdef \RR/(P(C)\ZZ)$, and we consider a basis $(\tau(s),\nu(s))$ such that $\tau(s)=y'(s)$, $\nu(s)=R_{-\pi/2}\tau(s)$, where $R_{-\pi/2}$ the rotation with angle $-\pi/2$. We assume that the parametrization is such that $\nu(s)$ is the outer unit normal to $C$.

The mapping $\varphi: (s,d)\mapsto y(s)+d\nu(s)$ is a $\Cder{1}$-diffeomorphism from $S\times \RR_+^*$ onto $\RR^2\setminus \overline{C}$, with
\begin{equation*}
  \frac{\partial \varphi}{\partial s}(s,d)=\tau(s) + d\kappa(s)\tau(s),\qandq  \frac{\partial \varphi}{\partial d}(s,d)= \nu(s),
\end{equation*}
where $\kappa(s)\geq 0$ is the curvature of $\partial C$ at $y(s)$.

In order to define a vector field $\overline{z}:\RR^2\setminus\overline{C}\rightarrow \RR^2$ such that $\divx \overline{z}=0$, it is sufficient to define a vector field $\tz:S\times \RR_+^*\rightarrow \RR^2$ such that $\overline{z}(x)=\tz(\varphi^{-1}(x))$ and
\begin{equation*}
  \Tr(D\tz D\varphi^{-1})=0.
\end{equation*}
In other words, we shall build a vector field $\tz$ such that
\begin{equation}\label{eq:divzzero}
  \frac{1}{1+\kappa(s)d}\partial_s z_1(s,d) +  \frac{\kappa(s)}{1+\kappa(s)d}z_2(s,d) +\partial_d z_2(s,d) =0.
\end{equation}
Here, for the sake of brevity, we have denoted by $\partial_s$ (resp. $\partial_d$) the derivatives with respect to $s$ (resp. $d$), and by $(z_1,z_2)$ the coordinates of $\tz$ in the basis  $(\tau(s),\nu(s))$.

Given $\alpha>0$ (to be fixed later), and the function $\eta: t\mapsto \min(t,2-t)$, we define
\begin{align}
  \label{eq:constrzoutu}  z_1(s,d)&= -\alpha \left(\int_0^s (\kappa(s')-\frac{2\pi}{P(C)})\d s'\right)\eta(d),\\
 \label{eq:constrzoutd}   z_2(s,d)&= \frac{1}{1+\kappa(s)d}\left(1+\alpha \left(\int_0^d \eta\right)\left(\kappa(s)- \frac{2\pi}{P(C)}\right)\right).
\end{align}
Observe that $\lim_{(s,d)\to (s_0,0)} z(s,d)=\nu(s_0)$, and that $z$ is continuous in $\RR^2\setminus \overline{C}$ since $\int_0^{P(C)}\kappa(s')\d s'=2\pi$.
Moreover, it is not difficult to check that $z$ satisfies~\eqref{eq:divzzero} as well.
As a result, $\divx \overline{z}=0$ pointwise in $\RR^2\setminus \overline{C}$, and since $\overline{z}$ is continuous we see by approximation that it also holds in the sense of distributions.

It remains to prove that $\abs{z}^2-1<0$.
\begin{align*}
  z_1^2+z_2^2-1&= \alpha^2 \left(\int_0^s (\kappa-\frac{2\pi}{P(C)})\right)^2\eta^2 \\
               &+ \frac{1}{(1+\kappa d)^2}\left[\left(1+\alpha \left(\int_0^d \eta\right)\left(\kappa- \frac{2\pi}{P(C)}\right)\right)^2 -1 - (\kappa d)^2-2\kappa d\right]\\
\end{align*}
There is a constant $M>0$ which only depends on $\sup_{\partial C} \kappa$ and $P(C)$ such that $\left(\int_0^s (\kappa-\frac{2\pi}{P(C)})\right)^2\leq M$ and $\abs{\kappa-\frac{2\pi}{P(C)}}^2\leq M$.

The term inside brackets is equal to 
\begin{align*}
  &\alpha^2 \left(\int_0^d \eta\right)^2\left(\kappa- \frac{2\pi}{P(C)}\right)^2 +2\alpha \left(\int_0^d \eta\right)\left(\kappa- \frac{2\pi}{P(C)}\right)- (\kappa d)^2-2\kappa d\\
  &\leq \alpha^2 M \left(\int_0^d \eta\right)^2 + 2\kappa \left( \alpha \left(\int_0^d \eta\right) -d \right) - \frac{4\pi \alpha}{P(C)} \int_0^d\eta - (\kappa d)^2 \\
  &\leq  2\kappa \left( \alpha \left(\int_0^d \eta\right) -d \right) +  \left(\alpha^2 M - \frac{4\pi \alpha}{P(C)}\right) \int_0^d\eta - (\kappa d)^2,
\end{align*}
since $\int_0^d\eta\leq 1$.
Hence, for $d\leq 1$, we obtain that for $\alpha$ small enough (depending on $M$ and $P(C)$), that term is less than or equal to
\begin{align*}
-\kappa d  - \frac{2\pi \alpha}{P(C)}\frac{d^2}{2} - (\kappa d)^2\leq 0,
\end{align*}
which yields (writing $K\eqdef \sup_{\partial C} \kappa$)
\begin{align}
  z_1^2+z_2^2-1&\leq \alpha^2 M d^2 + \frac{1}{1+K}\left(-\kappa d  - \frac{2\pi \alpha}{P(C)}\frac{d^2}{2} - (\kappa d)^2\right)\nonumber\\
               &\leq - \frac{1}{1+K}\left(\kappa d + \frac{\pi \alpha}{P(C)}\frac{d^2}{2}+(\kappa d)^2 \right) <0\label{eq:convexUpBound},
\end{align}
for $\alpha>0$ small enough (depending on $M$, $K$ and $P(C)$).

As for $d>1$, we may assume that $\alpha$ is small enough so that $\alpha \int_0^{+\infty}\eta \leq 1/2\leq d/2$. Moreover, $\int_0^d\eta\geq \int_0^1\eta=1/2$, so that the term inside brackets is less than or equal to
\begin{align*}
 -\kappa d- \frac{\pi \alpha}{P(C)} - (\kappa d)^2.
\end{align*}

\begin{align*}
  z_1^2+z_2^2-1&\leq  \alpha^2M -\frac{1}{1+\kappa d}  \left(\kappa d+ \frac{\pi \alpha}{P(C)} + (\kappa d)^2 \right)\\
               &\leq  \alpha^2M - \frac{1}{1+\kappa d}  \left(\kappa d+ \frac{\pi \alpha}{P(C)} \right)
\end{align*}
For $a\eqdef \frac{\pi \alpha}{P(C)}<1$, the mapping $x\mapsto -\frac{x+a}{x+1}$ is (strictly) decreasing on $[0,+\infty)$, hence upper bounded by $-a$, and we obtain that $z_1^2+z_2^2-1\leq \alpha^2M  -\frac{\pi \alpha}{P(C)}<0$ for $\alpha$ small enough.
\end{proof}

\begin{rem}\label{rem:decayz}
  A more straightforward construction would have been to construct $z$ parallel to the normals to $C$, or equivalently set $\alpha=0$ in~\eqref{eq:constrzoutu} and \eqref{eq:constrzoutd}. However, such a vector field would not decay in front of flat areas, where $\kappa(s)=0$, and we would have $\abs{z(s,d)}=1$ for all $d>0$. The above construction ``twists'' the field lines so as to obtain some decay of the norm. 
  
Still, the resulting upper bound~\eqref{eq:convexUpBound} for small $d$ depends on the local curvature of $\partial C$.
  If $\kappa(s)>0$, then, as $d\to 0^+$,
  \begin{equation*}
  \abs{z(s,d)}^2 \leq  1-\frac{\kappa(s)}{1+K}d + \smallo{d} 
  \end{equation*}
  On the other hand, if $\kappa(s)=0$, then
  \begin{equation*}
    \abs{z(s,d)}^2 \leq  1-\frac{\pi \alpha}{(1+K)P(C)}\frac{d^2}{2}.
  \end{equation*}
\end{rem}


\section{Numerical Illustrations}\label{sec:numerics}

In order to illustrate our theoretical findings, we have performed numerical computations on a discretized version of the denoising problem~\eqref{eq-rof}. Let us stress that this section does not provide any theoretical guarantees concerning the geometrical faithfulness of these approximations, and a careful study of the impact of discretization is an interesting avenue for future works. The code to reproduces these results can be found online\footnote{\url{https://github.com/gpeyre/2016-IP-tv-denoising/}}.

\subsection{Problem discretization}

The problem is discretized on an uniform grid $( (i/n,j/n) )_{i,j=0}^{n-1}$ of $n^2$ points in $[0,1]^2$. For simplicity, we use periodic boundary conditions. The input image $f$ is represented on this grid as $(f_{i,j})_{i,j=1}^n$ and is normalized so that $f_{i,j} \in [0,1]$. The recovered image $(u_{i,j})_{i,j}$ is defined on the same grid. Denoting $y \eqdef f+w$, the problem $\Pp_{\la}(y)$ is then approximated as
\eql{\label{eq-rof-discrete}\tag{$\Pp_\la^d(y)$}
	\umin{u \in \RR^{n^2}} \sum_{i,j=1}^n |u_{i,j}-y_{i,j}|^2 + \la \sum_{i,j}  \norm{\nabla_{i,j} f}
}
where $\norm{\cdot}$ is the Euclidean norm in $\RR^4$.
Here, $\nabla_{i,j} f \in \RR^4$ is a 4-fold discretization of the gradient operator, defined as
\eq{
\nabla_{i,j} f = n ( f_{i+1,j}-f_{i,j}, f_{i,j+1}-f_{i,j}, f_{i+1,j+1}-f_{i+1,j}, f_{i+1,j+1}-f_{i,j+1} ) \in \RR^4
}
so that the discrete gradient operator is $\nabla : u \in \RR^{n^2} \mapsto \nabla u \in \RR^{n^2 \times 4}$.
We also define the discrete divergence as
\eq{
	\diverg \eqdef -\nabla^\bot : \RR^{n^2 \times 4} \rightarrow \RR^{n^2}.
}
Note that this differs from the more usual forward  finite-difference approximation (used for instance in~\cite{chambolle2004algorithm}), and we found numerically that this improves the isotropy (rotation invariance) of the scheme. 

\subsection{Discrete dual problem and iterative algorithm}

We solve the finite dimensional convex optimization problem~\eqref{eq-rof-discrete} using the dual projected gradient descent of~\cite{chambolle2004algorithm}. 
It minimizes a discrete counterpart of~\eqref{eq-rof-dual}, which reads
\eql{\label{eq-dual-discr}\tag{$\Dd_\la^d(y)$}
	\umin{z \in \RR^{n^2 \times 4}} \enscond{
		\norm{ \frac{y}{\la} + \diverg(z) }_{\ell^2}
	}{
		z \in \Cc_\infty
	}	
}
\eq{
	\qwhereq
	\Cc_{\infty} \eqdef \enscond{z \in \RR^{n^2 \times 4}}{  
		\foralls (i,j), \norm{z_{i,j}} \leq 1
	}.
}
The solutions $z_\la$ of~\eqref{eq-dual-discr} are in general non-unique because the problem is not strictly convex, but the primal-dual relationship allow one to recovers the unique solution $u_\la$ of the primal problem~\eqref{eq-rof-discrete} as
\eq{
	u_\la \eqdef \diverg(z_\la) + \frac{y}{\la}. 
}

Starting by some initial $z^{(0)} \in \RR^{n^2 \times 4}$, the projected gradient descent reads
\eq{
	z^{(\ell+1)} \eqdef
	\text{Proj}_{\Cc_{\infty}} \pa{
		z^{(\ell)} + \tau \nabla ( \diverg(z) + \frac{y}{\la}), 	
	}
}
where the step size $\tau$ should satisfy $\tau < 2/\norm{\nabla}^2$ where $\norm{\nabla}^2 \le 16 n^2$ 
 is the operator norm of the gradient. 
The orthogonal projection on $\Cc_\infty$ is computed as
\eq{
	\tilde z = \text{Proj}_{\Cc_{\infty}}(z)
	\qwhereq
	\foralls (i,j), \quad
	\tilde z_{i,j} = \frac{z_{i,j}}{ \max( \norm{z_{i,j}}, 1 ) }.
}
The iterates converge $z^{(\ell)} \rightarrow z_\la$ toward a solution $z_{\la}$ of~\eqref{eq-dual-discr}, while the primal iterates
\eq{
	u^{(\ell)} \eqdef \diverg(z^{(\ell)}) + \frac{y}{\la} 
}
converge toward the unique solution $u_\la$ of~\eqref{eq-rof-discrete} with a speed $\norm{u_\la-u^{(\ell)}} = O(1/\ell)$ as shown in~\cite{fadili2010tv}.


\subsection{Denoising results}

Figure~\ref{fig-results} displays the solution $u_{\la}$ of~\eqref{eq-rof-discrete} for a set of increasing values of $\la$. We use here $n=512$, and the noise $w$ is a realization of a white noise where each pixel is Gaussian distributed with a variance $\si^2$ for $\si=0.2$.
As predicted by Theorem~\ref{thm:spt_stability}, this shows how the level sets of the solution progressively clusters around the extended support Ext$(Df)$ as $\la$ increases. In order to get some insight about the geometry of this extended support, we display the saturation points of $\norm{z_{0}} = (\norm{z_{0,i,j}})_{i,j}$. Note that since $z_0$ is non-unique and we used the one output by the discrete minimization scheme, we do not claim and theoretical guarantee about this procedure. In practice, we observed that stating the algorithm with $z^{(0)}=0$ leads to meaningful result about this extended support, that are shown on Figure~\ref{fig-results}.

\newcommand{\myfig}[1]{\includegraphics[width=0.16\textwidth]{figures/#1}}
\newcommand{\figRow}[4]{%
\myfig{#1/original.png}&%
\myfig{#1/noisy-img}&%
\myfig{#1/extended-support.png}&%
\myfig{#1/tv-noisy-#2-ls.png}&%
\myfig{#1/tv-noisy-#3-ls.png}&%
\myfig{#1/tv-noisy-#4-ls.png}\\}

\begin{figure}[H]
\begin{center}
\begin{tabular}{c@{\hspace{1mm}}c@{\hspace{1mm}}c@{\hspace{1mm}}c@{\hspace{1mm}}c@{\hspace{1mm}}c}
\figRow{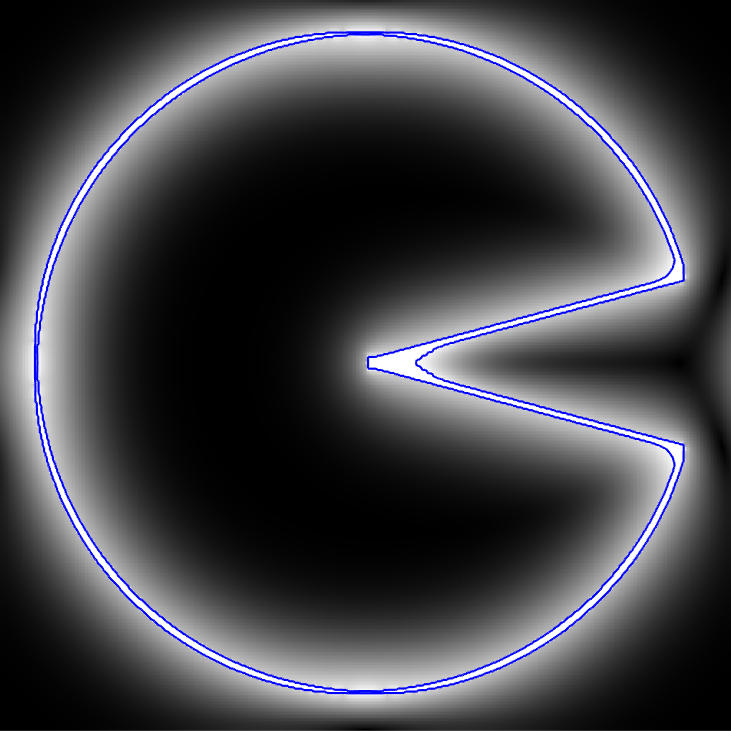}{7}{8}{9}
\figRow{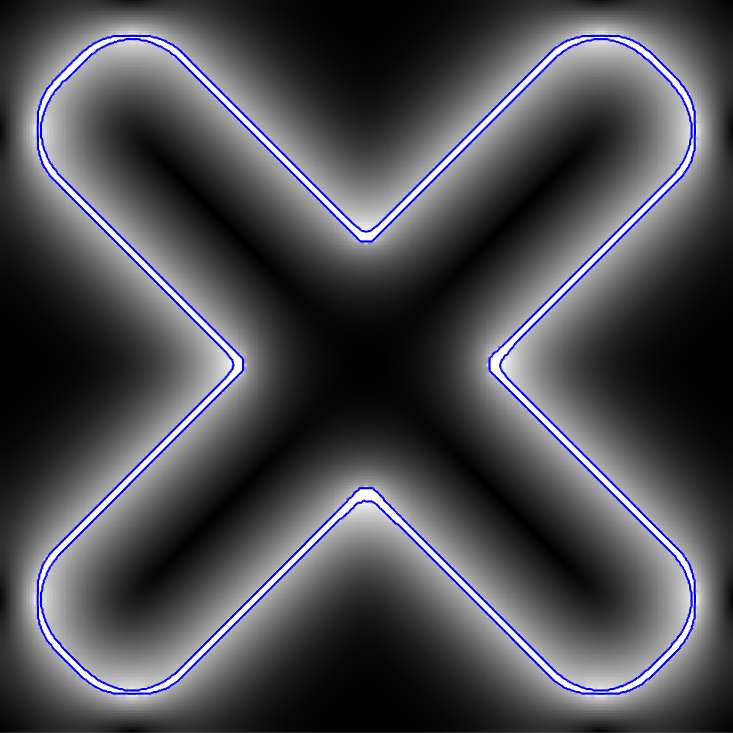}{8}{9}{10}
\figRow{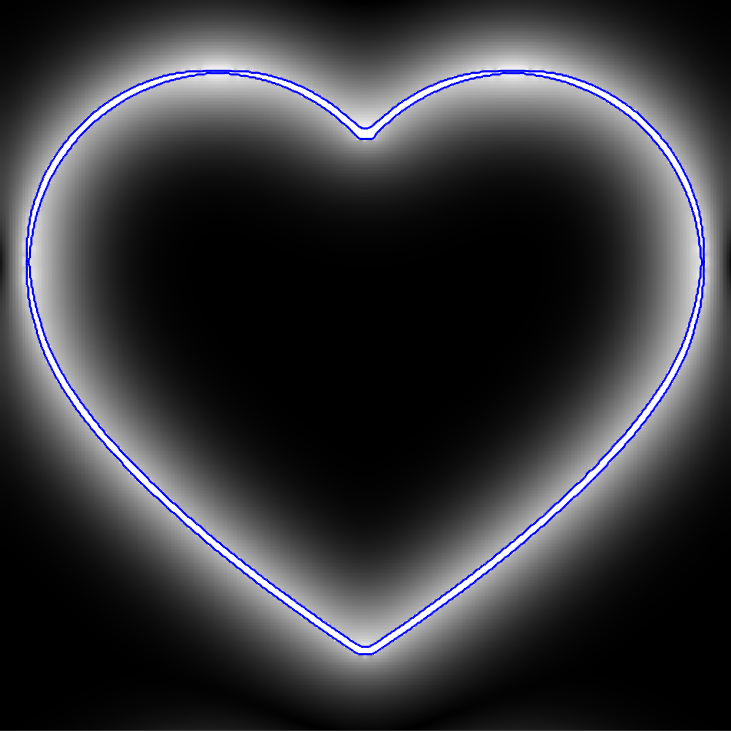}{8}{9}{10}
\figRow{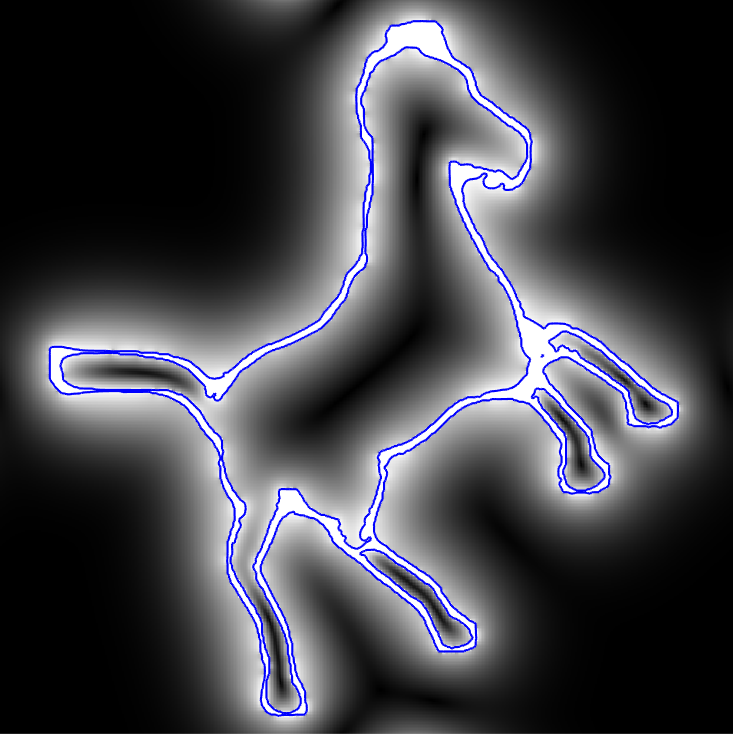}{8}{9}{10}
\figRow{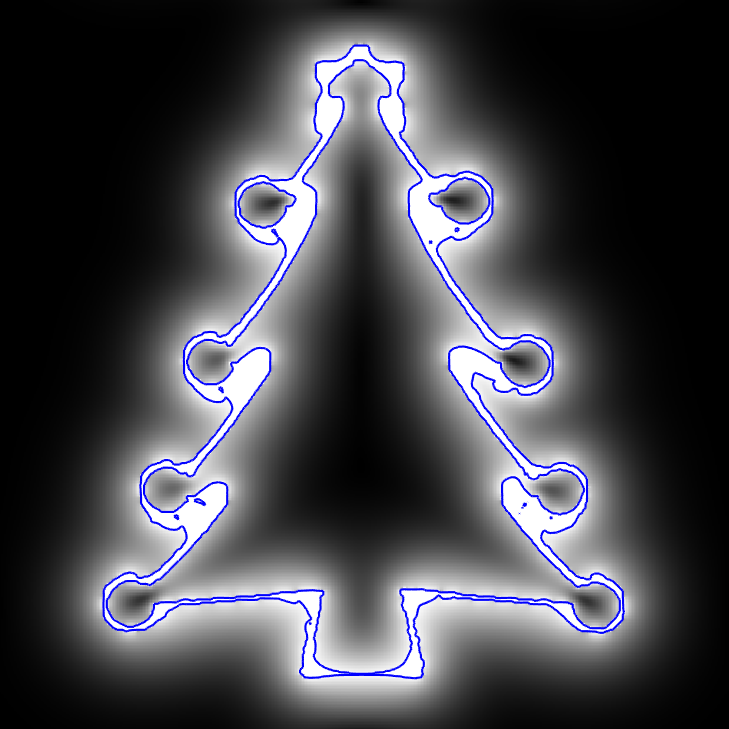}{8}{9}{10}
\figRow{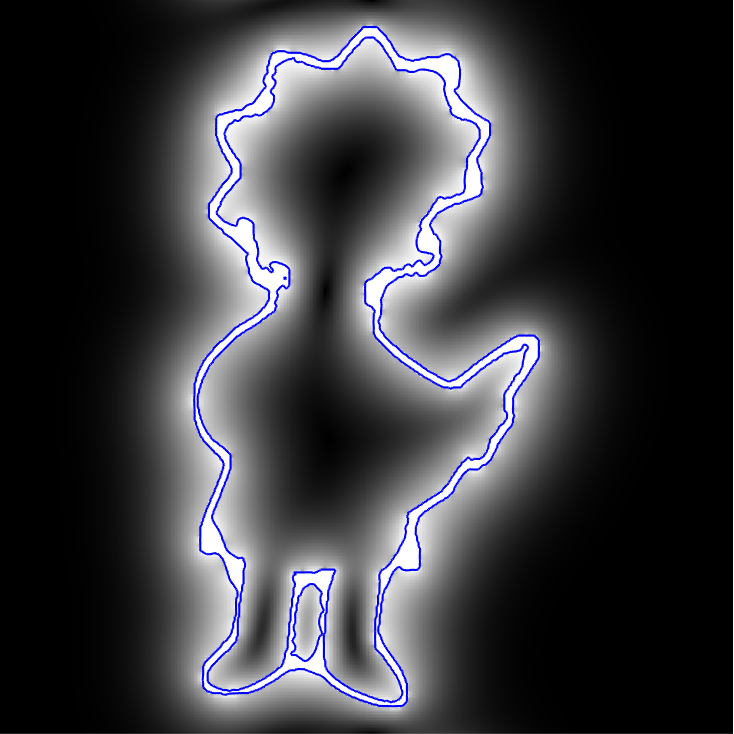}{8}{9}{10}
Original $f$ & 
Input $f+w$ & 
$\norm{z_0}$  &
$\la=0.1$ & 
$\la=0.15$ & 
$\la=0.2$  
\end{tabular}
\caption{\label{fig-results} 
Display of the discretized solution $u_{\la}$ of the discretized problem~\eqref{eq-rof-discrete} for several value of $\la$. The blue curves on top $u_{\la}$ of indicate the level sets of $u_{\la}$ (computed using bilinear interpolation on the grid). The blue curves on top of $\norm{z_0}$ indicate the obtained approximation of the boundary of the extended support Ext$(Df)$.
}
\end{center}
\end{figure}

\section*{Conclusion}
\label{sec-conclusion}

In this paper, we have characterized the regions in which the solutions to the two-dimensional TV denoising problem are  geometrically stable under $L^2$ additive noise. In particular, via the minimal norm certificate, we introduced the notion of an extended support and although the support of TV regularized solutions are in general not stable, we have proved that the support instabilities are confined to a neighbourhood of the extended support. We have also provided explicit examples of the extended support in the case of indicators of convex sets. Within the low noise regime, for the indicator set of a calibrable set $C$, the support of the solutions was shown to cluster around $\partial C$. While for indicator functions of general convex sets (including convex sets for which the source condition is not satisfied), the support of the  solutions was shown to cluster around the domain $C\setminus \mathrm{int}(C_{R^*})$, where $C_{R^*}$ is the maximal calibrable set inside $C$.

\section*{Acknowledgements}

Antonin Chambolle was partially supported by the ANR, programs ANR-12-BS01-0014-01 ``GEOMETRYA'' and ANR-12-IS01-0003 ``EANOI'' (joint with FWF No.~I1148). He acknowledges the hospitality of DAMTP and Churchill College (U. Cambridge) for the year 2015-2016.
Vincent Duval and Clarice Poon acknowledge support from the CNRS (D\'efi Imag'in de la Mission pour l'Interdisciplinarit\'e, project CAVALIERI).
The work of Gabriel Peyr\'e has been supported by the European Research Council (ERC project SIGMA-Vision).
Clarice Poon acknowledges support from the Fondation Sciences Math\'{e}matiques de Paris.

\printbibliography


\end{document}